\documentclass[12pt]{amsart}

\newtheorem{theorem}{Theorem}[section]
\newtheorem{lemma}[theorem]{Lemma}

\newtheorem{proposition}[theorem]{Proposition}

\usepackage[bookmarks]{hyperref}

\usepackage{graphicx,amscd,amsthm,amsfonts,amsopn,amssymb,verbatim,enumerate,xcolor}

\usepackage[letterpaper,left=1in,top=1in,right=1in,bottom=1in]{geometry}

\def\ta{\tilde a}
\def\tb{\tilde b}
\def\half{ \frac{1}{2}}
\def\D{\partial}

\def\R{{\mathbb R}}

\def\N{{\mathbb N}}

\def\nint{\mathop{\diagup\kern-13.0pt\int}}

\def\bas{\begin{align*}}
\def\eas{\end{align*}}
\def\bi{\begin{itemize}}
\def\ei{\end{itemize}}

\def\emph#1{{\it #1}}

\def\AA{{\mathcal A}}
\def\BB{{\mathcal B}}

\def\eps{{\epsilon}}

\theoremstyle{definition}
\newtheorem{remark}[theorem]{Remark}
\numberwithin{equation}{section}

\begin{document}

\title{The Dispersion Generalized Benjamin-Ono Equation}

\author{Albert Ai}
\address{Department of Mathematics, University of Wisconsin, Madison}
\email{aai@math.wisc.edu}

\author{Grace Liu}
\address{Department of Mathematics, University of California at Berkeley}
\email{graceliu0511@berkeley.edu}


\begin{abstract}
We consider the well-posedness of the family of dispersion generalized Benjamin-Ono equations. Earlier work of Herr-Ionescu-Kenig-Koch established well-posedness with data in $L^2$, by using a discretized gauge transform in the setting of Bourgain spaces. In this article, we remain in the simpler functional setting of Sobolev spaces, and instead combine a pseudodifferential gauge transform, a paradifferential normal form, and a variable coefficient Strichartz analysis to establish well-posedness in negative-exponent Sobolev spaces. Our result coincides with the classical well-posedness results obtained at the Benjamin-Ono and KdV endpoints.
\end{abstract}


\maketitle
\addtocontents{toc}{\protect\setcounter{tocdepth}{1}}
\tableofcontents

\section{Introduction}

In this article, we consider the Cauchy problem for the dispersion-generalized Benjamin-Ono equation,
\begin{equation}\label{BO}
(\D_t - |D|^\alpha \D_x) \phi = \half \D_x(\phi^2), \qquad \phi(0) = \phi_0,
\end{equation}
where $\phi: \R^{1 + 1} \rightarrow \R$, and $|D|^\alpha$ denotes the Fourier multiplier with symbol $|\xi|^\alpha$. The dispersion exponent $\alpha + 1$ may take a range of values; notably, \eqref{BO} corresponds to
\begin{itemize}
\item the classical Benjamin-Ono equation when $\alpha = 1$,
\item the KdV equation when $\alpha = 2$, and
\item the Burgers' equation when $\alpha = 0$.
\end{itemize}
In addition, \eqref{BO} has an order of dispersion reminiscent of the capillary-gravity water waves system when $\alpha = 1/2$, and the pure gravity water waves system when $\alpha = -1/2$. In the current article, we will be considering primarily the range $\alpha \in [1, 2]$ between the classical Benjamin-Ono and KdV equations.

\

The generalized Benjamin-Ono equation \eqref{BO} is Hamiltonian, with conserved quantities
\[
M(\phi) = \int \phi^2 \, dx, \qquad E(\phi) = \int \half \left||D|^\frac{\alpha}{2} \phi \right|^2 + \frac13 \phi^3 \, dx.
\]
Further, it has a scale invariance,
\[
\phi(t, x) \mapsto \lambda^\alpha \phi(\lambda^{\alpha + 1} t, \lambda x)
\]
with scale invariant Sobolev space $H^{\frac12 - \alpha}$. In particular, \eqref{BO} is $L^2$ critical for $\alpha = \frac12$ and energy critical for $\alpha = \frac13$. 

We recall that in the classical Benjamin-Ono setting with $\alpha = 1$, \eqref{BO} exhibits a quasilinear character, due to low-high frequency interactions aggravated by the derivative in the nonlinearity. In particular, the classical Benjamin-Ono equation satisfies only a continuous dependence on initial data, even at high regularity. It was proved by Molinet-Saut-Tzvetkov \cite{BOill} that this quasilinear character extends as well to the dispersion generalized setting, as soon as $\alpha < 2$. 

Extensive work has been done regarding the well-posedness of the Benjamin-Ono and KdV equations, corresponding to $\alpha = 1$ and $2$; see, respectively, \cite{ifrim2017well} and \cite{ifrim2022dispersive}, and the references therein. Also see Tao \cite{taobook} for a more complete discussion. Here, we highlight some key well-posedness thresholds and the corresponding methods, developed primarily in the course of the study of the classical Benjamin-Ono equation. 

We begin with the $H^1$ well-posedness for the classical Benjamin-Ono equation, a significant threshold obtained by Tao in \cite{taoBO} using a nonlinear \emph{gauge transformation}. By combining this gauge transformation with the use of $X^{s, b}$ spaces, Ionescu-Kenig \cite{IKbo} were able to prove the $L^2$ local (and hence global) well-posedness for the classical Benjamin-Ono equation. 

Several authors have since presented improved results using simplified proofs of the $L^2$ well-posedness. For instance, Molinet-Pilod \cite{MPcauchy} presented a simplified proof with a stronger unconditional uniqueness in $H^s$ for $s > \frac14$. More recently, Ifrim-Tataru \cite{ifrim2017well} provided another proof of $L^2$ well-posedness for the classical Benjamin-Ono using a two-part transformation combining paradifferential normal forms with the gauge transform, while avoiding the use of $X^{s, b}$ spaces. Finally, using the method of commuting flows and after Talbut's \cite{Tlowreg} work on conservation laws at negative regularities $H^s$, $s \in (-\half, 0)$, Killip-Laurens-Visan \cite{KLVsharp} established well-posedness of the classical Benjamin-Ono equation in $H^{s}$ for $s > -\frac12$.

For work regarding the family of dispersion generalized models, see \cite{molinet2018well} and the references therein. In particular, Herr \cite{herr} established the local well-posedness of \eqref{BO} for $\alpha \in (1, 2)$ in $H^s \cap \dot H^{\half - \frac{1}{\alpha}}$ with $s > \frac34(1 - \alpha)$. Here, the restriction to $\dot H^{\half - \frac{1}{\alpha}}$ may be viewed as a vanishing low frequency assumption on the initial data. Subsequently, Herr-Ionescu-Kenig-Koch \cite{herr2010differential} removed the low frequency assumption for the case of $L^2$ data, thus generalizing the earlier $L^2$ well-posedness result of Ionescu-Kenig for classical Benjamin-Ono to the range of dispersions $\alpha \in (1, 2)$. Their approach uses a discretized pseudodifferential gauge transform, combined with $X^{s, b}$ spaces. For the low-dispersion range $\alpha \in (0, 1)$, Molinet-Pilod-Vento \cite{molinet2018well} established local well-posedness in $H^s$ with $s > \frac32 - \frac{5\alpha}{4}$, using a modified energy approach combined with $X^{s,b}$ spaces, but avoiding the use of a gauge transform.

\

In the present article, our objective is to prove the following well-posedness theorem:

\begin{theorem}\label{t:lwp}
The generalized Benjamin-Ono equation \eqref{BO} with $\alpha \in [1, 2]$ is locally well-posed in $H^s$ with $s > \frac34(1 - \alpha)$.
\end{theorem}

On one hand, this extends the well-posedness of Herr-Ionescu-Kenig-Koch \cite{herr2010differential} to negative Sobolev regularities. On the other, one may view this as the removal of the vanishing low frequency assumption in the result of Herr \cite{herr}.

We use a two-part transformation involving a paradifferential normal form combined with a pseudodifferential gauge transform, in the spirit of the approach used by Ifrim-Tataru as applied in the context of the classical Benjamin-Ono equation. This approach is particularly well-suited to the pseudodifferential context, because in contrast with previous attempts, it avoids the use of $X^{s, b}$ spaces. As observed by Herr-Ionescu-Kenig-Koch, the control of the pseudodifferential gauge transform in $X^{s, b}$ spaces presents a substantial challenge, even after the discretization employed there.

However, unlike in the classical Benjamin-Ono setting where the gauge transform is purely multiplicative, the pseudodifferential variant of the gauge transform is still not bounded on $L^p$ spaces. To address this, we exclude the unbounded component of the gauge transform, which corresponds to the component of the nonlinearity consisting of interactions between very low and high frequencies. The consequence of this exclusion is that, instead of using energy and dispersive estimates for a constant coefficient linear flow, we will need to conduct the linear analysis on a variably transported background, though at very low frequency. In particular, the main difficulty will be adapting the linear and bilinear Strichartz estimates to this setting of a transported background.

We may view our approach as a microlocal division of the nonlinearity of \eqref{BO} into three components, corresponding respectively to a paradifferential normal form, a pseudodifferential gauge transform, and a perturbation of the linear flow. One significant advantage of this perspective is that the treatment of each of these components is essentially independent. A second advantage, preserved from the work of Ifrim-Tataru, is that we are able to remain in the simpler Sobolev functional setting.

Our paper is organized as follows. In the next two sections, Sections~\ref{s:transport-strich} and \ref{s:bilinear}, we perform the linear analysis, proving Strichartz and bilinear estimates in the presence of a variable transport. In Section~\ref{s:nf}, we present the normal form analysis, introducing the pseudodifferential conjugation and paradifferential normal form. We also establish bounds for both transformations. Finally, we present the bootstrap and well-posedness arguments in the last three sections.

\subsection{Acknowledgements}

The first author was supported by the NSF grant DMS-2220519 and the RTG in Analysis and Partial Differential equations grant DMS-2037851.

The authors would like to thank Mihaela Ifrim and Daniel Tataru for many helpful discussions.

\section{Strichartz estimates with transport}\label{s:transport-strich}

In this section, our objective is to prove Strichartz estimates for a linear evolution equation of the form
\begin{equation}\label{general}
\begin{cases}
(i\D_t + A^w(t, x, D))u = f, \qquad &\text{in } (0, 1) \times \R, \\
u(0) = u_0, \qquad &\text{on } \R,
\end{cases}
\end{equation}
where $A^w$ denotes the self-adjoint Weyl quantization of a symbol $a(t, x, \xi)$ which has a specific form as the sum of a constant coefficient dispersive term with a variable transport. Precisely, we consider a symbol
\begin{equation}\label{transport-symbol}
a(t, x, \xi) = (b(t, x)\xi + |\xi|^m) \chi_\lambda(\xi),
\end{equation}
where $\chi_\lambda$ is the $\lambda$-supported symbol of a Littlewood-Paley dyadic partition of unity. We assume that $b$ satisfies, with $\delta = \frac{2 - m}{2}$,
\begin{equation}\label{L1-symbol-b}
\begin{aligned}
&\|\D_x^\alpha \D_t^\gamma b\|_{L^1_t([0, 1]; L^\infty)} \lesssim \lambda^{\delta(|\alpha| - 1)}, \qquad &|\alpha| \geq 1, \gamma \in \{0, 1\},
\end{aligned}
\end{equation}
and
\begin{equation}\label{Linfty-symbol-b}
\|\D_t^\gamma b\|_{L^\infty_t([0, 1]; L^\infty)} \lesssim 1, \qquad  \gamma \in \{0, 1\}.
\end{equation}

In this section, we prove the following Strichartz and lateral Strichartz estimates:
\begin{theorem}\label{t:strichartz-transport}
Let $m \in [2, 3]$, $\delta = \frac{2 - m}{2}$, and $a(t, x, \xi)$ given by \eqref{transport-symbol} satisfy \eqref{L1-symbol-b} and \eqref{Linfty-symbol-b}. Let $u$ have frequency support $\lambda$ and solve \eqref{general}. Then for $p, q$ satisfying
\begin{equation}
\frac{2}{p} + \frac{1}{q} = \frac{1}{2}, \quad 2 \leq p \leq \infty, \quad 1 \leq q \leq \infty,
\end{equation}
we have
\begin{equation}\label{direct-strich}
\|u\|_{L^p([0, 1];L^q)} \lesssim \lambda^{\frac{2\delta}{p}} (\|u\|_{L^\infty([0, 1];L^2)} + \|f\|_{L^1([0, 1];L^2)}),
\end{equation}
as well as the lateral Strichartz estimate,
\begin{equation}\label{lat-strich}
\|u\|_{L^p_x (L^q_t([0, 1]))} \lesssim \lambda^{\frac14 + \frac{1}{q}(\half - m)}(\|u\|_{L^\infty([0, 1];L^2)} + \|f\|_{L^1([0, 1];L^2)}).
\end{equation}
\end{theorem}

Due to the variable transport $b$, we use a physical space approach, constructing a wave packet parametrix. A delicate argument is needed to obtain estimates on the unit time scale, since the transport term is nonperturbative, even at low frequency. Addressing this requires the use of an exact eikonal phase function in the packet, rather than its linearization.

\begin{remark}
The Strichartz estimates presented in the analysis of the gravity-capillary water waves \cite{capillary} apply in the context of the class $L^1S_{1, \delta}^{m, (k)}(\lambda)$ of symbols satisfying
\begin{equation}\label{L1-symbol}
\begin{aligned}
&\|\D_x^\alpha \D_\xi^\beta a\|_{L^1_t([0, 1]; L^\infty)} \lesssim \lambda^{m - |\beta| + \delta(|\alpha| - k)}, \qquad &|\alpha| \geq k.
\end{aligned}
\end{equation}
However, unless $m \leq 1$, dispersive estimates for such symbols requires the use of microlocal time scales, even for smooth symbols with large $k$. Precisely, the dispersive estimates require symbols $a$ satisfying
\[
\lambda^{2m - 2} a \in L^1S_{1, \delta}^{m, (2)}(\lambda).
\]
In our current context, we would like to avoid the use of microlocal scales, since this leads to derivative losses. This is possible due to our additional assumptions on the time derivatives of $b$, corresponding to the cases of \eqref{L1-symbol-b} and \eqref{Linfty-symbol-b} with $\gamma = 1$.

\end{remark}

\begin{remark}
Using a classical Hadamard parametrix, Alazard-Burq-Zuily \cite{alazard2011strichartz} established the standard Strichartz estimate \eqref{direct-strich} in the context of the capillary water wave equations. In this setting, the linear evolution consists of a variable transport and a dispersive term of order $m = 3/2$.
\end{remark}

\subsection{The Hamilton flow}

We consider a rescaling of $a$ and $b$,
\begin{equation}\label{transport-symbol-scaled}
\tilde a(t, x, \xi) = \tau a(\tau t, \mu x, \mu^{-1} \xi), \qquad \tilde b = \tau \mu^{-1} b(\tau t, \mu x), \qquad \mu = \tau^{\frac12} \lambda^{-\delta},
\end{equation}
for arbitrary fixed $\tau \in [\lambda^{-m}, 1]$, and where $\delta = \frac{2 - m}{2}$ denotes the exponent of the frequency scale $\delta \xi$ of a wave packet. Note that the new frequency $\xi$ after this rescaling is 
\[
\xi \approx \mu \lambda = (\tau \lambda^m)^\half.
\]
 
We first observe some basic estimates on the rescaled symbols $\ta, \tb$. We have for $|\alpha| \geq 1$ the counterpart to \eqref{L1-symbol-b},
\begin{equation}\label{tb-est}
\|\D_x^\alpha\D_t^\gamma \tilde b\|_{L^1L^\infty} \lesssim \tau \mu^{|\alpha| - 1} \lambda^{\delta(|\alpha| - 1)} = \tau^{\frac{|\alpha| + 1}{2} + |\gamma|}.
\end{equation}
Using this to establish estimates for $\ta$, we have for $|\alpha| \geq 1$ and $|\beta| \geq 1$,
\begin{equation}\label{mixed}
\|\D_x^\alpha\D^\beta_\xi \tilde a\|_{L^1L^\infty} \lesssim \tau^{\frac{|\alpha| + 1}{2}}.
\end{equation}
For $|\beta| \geq 2$, we have
\begin{equation}\label{xixi}
\begin{aligned}
\|\D_\xi^\beta \tilde{a}\|_{L^1_t([0, 1]; L^\infty)} &\lesssim \mu^{- |\beta|} \tau\|\D_\xi^\beta a\|_{L^\infty_t([0, \tau]; L^\infty)} \\
&\lesssim \mu^{- |\beta|} \tau \lambda^{m - |\beta|} \\
&= (\tau^{-1}\lambda^{-m})^{\frac12 (|\beta| - 2)} \lesssim 1.
\end{aligned}
\end{equation}

\

Our objective in this subsection is to establish estimates for the Hamilton characteristics associated to $\ta$,
\begin{equation}\label{ham}
\begin{aligned}
&\dot{x}(t) = \ta_\xi(t, x(t), \xi(t)), \\
&\dot{\xi}(t) = -\ta_x(t, x(t), \xi(t)), \\
&(x, \xi)(0) = (x, \xi).
\end{aligned}
\end{equation}
We denote the solution $(x(t), \xi(t))$ to \eqref{ham} at time $t$ with initial data $(x, \xi)$ by
\[
(x^t, \xi^t) = (x^t(x, \xi), \xi^t(x, \xi)).
\]

Throughout, we will consider $\xi$ at the rescaled frequency, such that $\xi \approx \mu\lambda$. We begin by establishing the following basic estimates on the flow $(x^t, \xi^t)$.

\begin{lemma}\label{l:hamEst}
Consider $(x^t, \xi^t)$ satisfying \eqref{ham} with $\xi \approx \mu\lambda$. Then for $t \in [0, \tau^{-1}]$, 
\[
\xi^t \approx \xi \approx \mu\lambda, \qquad \|\dot{\xi}^t\|_{L^1L^\infty} \lesssim \tau \mu\lambda,
\]
and 
\[
\dot{x}^t \approx \mu\lambda, \qquad \|\ddot{x}^t\|_{L^1L^\infty} \lesssim  \tau \mu\lambda.
\]
\end{lemma}

\begin{proof}

For the estimates on $\xi^t$, we have
\[
\dot{\xi}^t = -\ta_x(t, x^t, \xi^t) = -\tb_x(t, x^t) \xi^t
\]
and apply Gronwall's inequality with \eqref{tb-est}.

\

For the estimates on $x^t$, we write
\[
\dot{x}^t = \ta_\xi(t, x^t, \xi^t) = \tb(t, x^t) + \tau\mu^{- m}m|\xi^t|^{m - 1} \approx O(\tau \mu^{-1}) + \tau \mu^{-m}(\mu \lambda)^{m - 1} \approx \mu \lambda.
\]
Further, we have using \eqref{Linfty-symbol-b}, \eqref{mixed}, and \eqref{xixi},
\[
\|\ddot{x}^t\|_{L^1L^\infty} = \|\tb_t(t, x^t) + \tilde a_{x\xi} \dot{x}^t + \ta_{\xi\xi} \dot{\xi}^t\|_{L^1L^\infty} \lesssim  \tau \mu\lambda.
\]

\end{proof}

Next, we show that the flow \eqref{ham} is bilipschitz, which will be central to the coherence of the wave packet parametrix construction.
\begin{proposition}\label{p:bilip}
Consider $(x^t, \xi^t)$ satisfying \eqref{ham} with $\xi \approx \mu\lambda$. 

\begin{enumerate}[a)]

\item We have
\[
|\D_x x^t| + |\D_x \xi^t| \lesssim 1.
\]
\item Further, for $t \ll \tau^{-1}$,
\[
 |\D_x x^t - 1| \leq \half.
 \]
\end{enumerate}

\end{proposition}
\begin{proof}

 Differentiating \eqref{ham}, we have
\begin{equation*}
\begin{aligned}
\frac{d}{dt}\D_x x^t &= \ta_{\xi x}(t, x^t, \xi^t) \D_x x^t + \ta_{\xi \xi}(t, x^t, \xi^t) \D_x \xi^t, \\
\frac{d}{dt}\D_x \xi^t &= -\ta_{xx}(t, x^t, \xi^t) \D_x x^t - \ta_{x\xi}(t, x^t, \xi^t) \D_x \xi^t.
\end{aligned}
\end{equation*}
Substituting the form of $\tilde a$, we obtain
\begin{equation}\label{linHam}
\begin{aligned}
\frac{d}{dt}\D_x x^t &= \tb_x(t, x^t) \D_x x^t + \tau \mu^{-m}m(m - 1) |\xi^t|^{m - 2}\D_x \xi^t, \\
\frac{d}{dt}\D_x \xi^t &= - \tb_{xx}(t, x^t) \xi^t \D_x x^t - \tb_x(t, x^t) \D_x \xi^t.
\end{aligned}
\end{equation}
The coefficients on the right hand side of the first equation are bounded by \eqref{mixed} and \eqref{xixi}. In the second equation for $\D_x\xi^t$, the second term on the right hand side likewise has bounded coefficient, but the first term does not. To control this term, we integrate $\tb$. Observing that
\[ \frac{d}{dt} \tb_x(t, x^t) = \tb_{tx}(t, x^t) + \tb_{xx}(t, x^t) \dot{x}^t, \]
we may rewrite this first term as
\begin{equation*}
\begin{aligned}
\tb_{xx}(t, x^t)\xi^t \D_x x^t &= \tb_{xx}(t, x^t)\dot{x}^t \cdot \frac{\xi^t}{\dot{x}^t} \D_x x^t \\
&= \frac{d}{dt}\left( \tb_x(t, x^t) \frac{\xi^t}{\dot{x}^t} \D_x x^t\right)  \\
&\quad - \tb_x(t, x^t)\frac{d}{dt}\left(\frac{\xi^t}{\dot{x}^t}\right)\D_x x^t - \tb_x(t, x^t)\frac{\xi^t}{\dot{x}^t}\frac{d}{dt}\D_x x^t - \tb_{tx}(t, x^t)\frac{\xi^t}{\dot{x}^t} \D_x x^t \\
&= \frac{d}{dt}\left( \tb_x(t, x^t) \frac{\xi^t}{\dot{x}^t} \D_x x^t\right)  \\
&\quad - \tb_x(t, x^t)\left(\frac{\dot{\xi}^t}{\dot{x}^t} - \frac{\xi^t}{(\dot{x}^t)^2}\ddot{x}^t\right)\D_x x^t- \tb_{tx}(t, x^t)\frac{\xi^t}{\dot{x}^t} \D_x x^t  \\
&\quad - \tb_x(t, x^t)\frac{\xi^t}{\dot{x}^t}\left(\tb_x(t, x^t) \D_x x^t + \tau m(m - 1) \mu^{-m}|\xi^t|^{m - 2}\D_x \xi^t \right).
\end{aligned}
\end{equation*}
The coefficients on the terms of the second and third rows of the right hand side are bounded by \eqref{tb-est}. We conclude by an application of Gronwall's inequality that
\[
|\D_x x^t| + |\D_x \xi^t| \lesssim 1.
\]
In turn, substituting this into \eqref{linHam}, we obtain
\[
|\D_x x^t - 1| \leq \half
\]
for $t \ll \tau^{-1}$.

\end{proof}

We similarly prove that the flow satisfies a dispersive property:

\begin{proposition}\label{p:disp}
Consider $(x^t, \xi^t)$ satisfying \eqref{ham} with $\xi \approx \lambda$. 

\begin{enumerate}[a)]
\item We have
\[
|\D_\xi x^t| + |\D_\xi \xi^t| \lesssim 1.
\]
\item For $t \ll \tau^{-1}$,
\[
 |\D_\xi \xi^t - 1| \leq \half.
 \]
\item For $t \ll \tau^{-1}$,
\[\D_\xi x^t \approx t.\]
\end{enumerate}

\end{proposition}

\begin{proof}

The estimates of $a)$ and $b)$ are proven in the same way as the corresponding estimates of Proposition~\ref{p:bilip}. 

\

For $c)$, we begin with the counterpart to the first equation of \eqref{linHam},
\begin{equation*}
\begin{aligned}
\frac{d}{dt}\D_\xi x^t &= \tb_x(t, x^t) \D_\xi x^t + \tau \mu^{-m} m(m - 1)|\xi^t|^{m - 2}\D_\xi \xi^t.
\end{aligned}
\end{equation*}
Consider the second term on the right hand side. We have by $b)$, and the estimate on $\xi^t$ of Lemma~\ref{l:hamEst},
\[
 \tau\mu^{-m} |\xi^t|^{m - 2}\D_\xi \xi^t \approx \tau\mu^{-m}(\mu\lambda)^{m - 2} = 1.
\]
Using this with part $a)$ and the estimates on $\tilde b_x$ of \eqref{tb-est}, we obtain $c)$.
\end{proof}

\subsection{The eikonal equation}

The Hamilton flow \eqref{ham} forms the characteristics of solutions to the eikonal equation,
\begin{equation}\label{eik}
\D_t \psi_{x, \xi}(t, y) = -\ta(t, y, \D_y \psi_{x, \xi}(t, y)), \qquad \psi_{x, \xi}(0, y) = \xi(y - x),
\end{equation}
which will serve as the phase of our wave packets. In preparation, we use the regularity of the characteristics to establish estimates on the eikonal solutions.

\begin{lemma}\label{l:eikonal}
Consider a solution $\psi_{x, \xi}$ to \eqref{eik} with $\xi \approx \lambda$. We have
\[
\D_y\psi_{x, \xi} \approx \xi, \qquad |\D_y^2 \psi_{x, \xi}| \lesssim 1.
\]
\end{lemma}

\begin{proof}
We have
\[ \D_y\psi_{x, \xi}(t, x^t(x, \xi)) = \xi^t(x, \xi) \]
so that using Lemma~\ref{l:hamEst}, 
\[
\D_y\psi_{x, \xi} \approx \xi.
\]
Then using Proposition~\ref{p:bilip},
\[
|\D_y^2 \psi_{x, \xi}| \approx |\D_x \xi^t| \lesssim 1.
\]
\end{proof}

Next, we establish the higher regularity of $\psi_{x, \xi}$:

\begin{proposition}\label{p:eikonaly}
We have $|\D^\alpha_y \psi_{x, \xi}| \lesssim 1$ for $|\alpha| \geq 2$.
\end{proposition}

\begin{proof}

Differentiating \eqref{eik}, we have
\begin{equation*}
\begin{aligned}
-\D_t \D_y \psi_{x, \xi} &= \ta_y(t, y,  \D_y \psi_{x, \xi}) + \ta_\xi(t, y,  \D_y \psi_{x, \xi}) \D_y^2 \psi_{x, \xi}.
\end{aligned}
\end{equation*}
Differentiating again and writing $v =  \D_y^2 \psi_{x, \xi}$, we have
\begin{equation}\label{d2eikonal}
\begin{aligned}
-\D_t v &= \ta_{yy}(t, y,  \D_y \psi_{x, \xi}) + 2\ta_{y\xi}(t, y, \D_y \psi_{x, \xi})v + \ta_{\xi\xi}(t, y,  \D_y \psi_{x, \xi})v^2 \\
&\quad + \ta_\xi(t, y, \D_y \psi_{x, \xi}) \D_y v.
\end{aligned}
\end{equation}

The first term on the right hand side of \eqref{d2eikonal} is unbounded. To address this, we integrate:
\begin{equation*}
\begin{aligned}
\ta_{yy}(t, y,  \D_y \psi_{x, \xi}) &= \ta_\xi(t, y, \D_y \psi_{x, \xi}) \tb_{yy}(t, y) \frac{\D_y \psi_{x, \xi}}{\ta_\xi(t, y, \D_y \psi_{x, \xi})} \\
&= \ta_\xi(t, y, \D_y \psi_{x, \xi}) \D_y \left( \tb_{y}(t, y) \frac{\D_y \psi_{x, \xi}}{\ta_\xi(t, y, \D_y \psi_{x, \xi})} \right) \\
&\quad - \tb_{y}(t, y) v + \tb_{y}(t, y)^2  \left( \frac{\D_y \psi_{x, \xi}}{\ta_\xi(t, y, \D_y \psi_{x, \xi})} \right).
\end{aligned}
\end{equation*}

Substituting into \eqref{d2eikonal}, we obtain
\begin{equation}\label{d2eikonal2}
\begin{aligned}
-\D_t v &= \ta_{\xi\xi}(t, y,  \D_y \psi_{x, \xi})v^2 \\
&\quad + (\tb_{y}(t, y) + \ta_\xi(t, y, \D_y \psi_{x, \xi}) \D_y) \left( v + \tb_{y}(t, y) \frac{\D_y \psi_{x, \xi}}{\ta_\xi(t, y, \D_y \psi_{x, \xi})} \right).
\end{aligned}
\end{equation}
Defining
\[
\tilde v = \tb_{y}(t, y) \frac{\D_y \psi_{x, \xi}}{\ta_\xi(t, y, \D_y \psi_{x, \xi})}, \qquad
w = v + \tilde v,
\]
we have
\begin{equation}\label{d2eikonal3}
\begin{aligned}
(\D_t + \tilde a_{\xi\xi}(t, y,  \D_y \psi_{x, \xi})v + \tilde b_{y}(t, y) + \tilde a_\xi(t, y, \D_y \psi_{x, \xi}) \D_y)  w = \D_t \tilde v + \tilde a_{\xi\xi}(t, y,  \D_y \psi_{x, \xi})v  \tilde v.
\end{aligned}
\end{equation}
We see that we have a transport equation for $w$, where the right hand side is bounded, using Lemma~\ref{l:eikonal} for the estimate $|v| \lesssim 1$. To address the unbounded transport velocity $\tilde  a_\xi$, we apply the Galilean transformation
\[
u(t, y) = w(t, y + \tau \mu^{-m} m \xi^{m - 1} t)
\]
which satisfies, by applying the same Galilean transformation to the equation,
\begin{equation}\label{d2eikonal4}
\begin{aligned}
(\D_t + a_{\xi\xi} v + b_{y} + (a_\xi - \tau \mu^{-m} m \xi^{m - 1} ) \D_y) u = \D_t \tilde v + \tilde a_{\xi\xi} v  \tilde v.
\end{aligned}
\end{equation}

The argument above has been applied for the case $\alpha = 2$ (which was already established in Lemma~\ref{l:eikonal}). The cases $\alpha > 2$ then follow by differentiating \eqref{d2eikonal3} to arrive at a transport equation of the same form. 
\end{proof}

We also require regularity of $\psi_{x, \xi}$ with respect to the initial data $(x, \xi)$:

\begin{proposition}\label{p:eikonalx}
We have $|\D^\alpha_{y, x, \xi} \psi_{x, \xi}| \lesssim 1$ for $|\alpha| \geq 2$. Further, we have 
\[|\D_\xi \psi_{x, \xi}| \lesssim 1 + |y - x^t|, \quad|\D_x \psi_{x, \xi} - \xi|\lesssim 1.
\]
\end{proposition}

\begin{proof}
The proof is similar to the proof of Proposition \ref{p:eikonaly}. For instance, differentiating \eqref{eik}, we have
\begin{equation*}
\begin{aligned}
-\D_t \D_x \psi_{x, \xi} &= \ta_\xi(t, y,  \D_y \psi_{x, \xi}) \D_y \D_x \psi_{x, \xi}.
\end{aligned}
\end{equation*}
Then apply the Galilean transformation as in the proof of Proposition \ref{p:eikonaly}.
\end{proof}

\subsection{The parametrix construction}

Following \cite{tataru2004phase}, we use a Fourier–Bros–Iagolnitzer (FBI) phase space transform to construct a wave packet parametrix. For a more thorough discussion of the properties of the FBI transform, the reader may refer to the comprehensive exposition of Delort \cite{FBId}.

The FBI transform takes the form 
\[(Tf)(x, \xi) = 2^{-\frac{d}{2}} \pi^{-\frac{3d}{4}} \int e^{-\half(x - y)^2}e^{i\xi(x - y)} f(y)\, dy,\]
and is an isometry from $L^2(\R^d)$ to phase space $L^2(\R^{2d})$ with an inversion formula
\[f(y) = (T^*Tf)(y) = 2^{-\frac{d}{2}} \pi^{-\frac{3d}{4}} \int e^{-\half(x - y)^2}e^{-i\xi(x - y)} (Tf)(x, \xi) \, dx d\xi.\]

We can use the FBI transform to quantify the phase space localization of the evolution operator $S(t, s)$ around the corresponding Hamilton flow. Let $\chi(t, s)$ denote the family of transformations on the phase space $L^2(\R^{2d})$ given by \eqref{ham},
$$\chi(t, s)(x^s, \xi^s) = (x^t, \xi^t).$$
It was shown in \cite{tataru2004phase} that for the class of symbols $a \in S_{0, 0}^{0, (k)}$ defined by
\begin{equation}\label{S2}
|\D_x^\alpha \D_\xi^\beta a(t, x, \xi)| \leq c_{\alpha ,\beta}, \qquad |\alpha| + |\beta| \geq k,
\end{equation}
the flow satisfies the following properties:
\begin{theorem}\label{t:bilipthm}
Let $a(t, x, \xi) \in S_{0, 0}^{0, (2)}$. Then
\begin{enumerate}
\item The Hamilton flow \eqref{ham} is well-defined and bilipschitz.

\item The kernel $\tilde{K}(t,s)$ of the phase space operator $TS(t, s)T^*$ decays rapidly away from the graph of the Hamilton flow,
\[
|\tilde{K}(t, x, \xi, s, y, \eta)| \lesssim (1 + |(x, \xi) - \chi(t, s)(y, \eta)|)^{-N}.
\]
\end{enumerate}
\end{theorem}
Then we have the following phase space representation for solutions to (\ref{general}), as a consequence of \cite[Theorem 4]{tataru2004phase}:
\begin{theorem}\label{t:repformula}
Let $a(t, x, \xi) \in S_{0, 0}^{0, (2)}$. Then the kernel $K(t, s)$ of the evolution operator $S(t, s)$ for $i\D_t + A^w$ can be represented in the form 
$$K(t, y, s, \tilde{y}) = \int e^{-\half(\tilde{y} - x^s)^2} e^{-i\xi^s(\tilde{y} - x^s)} e^{i(\psi(t, x, \xi) - \psi(s, x, \xi))} e^{i\xi^t(y - x^t)} G(t, s, x, \xi, y) \, dx d\xi ,$$
where the function $G$ satisfies
\[
|(x^t - y)^\gamma \D_x^\alpha \D_\xi^\beta \D_y^\nu G(t, s, x, \xi, y)| \lesssim c_{\gamma, \alpha, \beta, \nu}.
\]
\end{theorem}

Theorems \ref{t:bilipthm} and \ref{t:repformula} were generalized in \cite{marzuola2008wave} to the class of symbols $a\in S^{(k)}L_\chi^1$ with $k = 2$, satisfying
\[
\sup_{x, \xi} \int_0^1|\D_x^\alpha \D_\xi^\beta a(t, \chi(t, 0)(x, \xi))| \, dt \leq c_{\alpha, \beta}, \qquad |\alpha| + |\beta| \geq k.
\]

We will not be able to directly use the parametrix of Theorem~\ref{t:repformula} in our current setting, because our symbol $\ta$ does not fall in the symbol class $S_{0, 0}^{0, (2)}$ due to the variable transport $\tb$. Instead, we will prove an appropriate adaptation of Theorem \ref{t:repformula} for our symbol directly. Rather than using the $(x^s, \xi^s)$-centered bump functions
\[ e^{-\half(\tilde{y} - x^s)^2} e^{-i\xi^s(\tilde{y} - x^s)} e^{ - i\psi(s, x, \xi)}  \] 
as wave packets, we refine the phase using solutions to the eikonal equation \eqref{eik},
\begin{equation*}
\D_t \psi_{x, \xi}(t, y) = -\tilde a(t, y, \D_y \psi_{x, \xi}(t, y)), \qquad \psi_{x, \xi}(0) = \xi(y - x),
\end{equation*}
and use instead the packets
\[ e^{-\half(\tilde{y} - x^s)^2} e^{-i\psi_{x, \xi}(s, y)}. \]

\begin{theorem}\label{t:repformulaEik}
The kernel $K(t, s)$ of the evolution operator $S(t, s)$ for $i\D_t + \tilde A^w$, where $\ta$ is given by \eqref{transport-symbol}, \eqref{L1-symbol-b}, \eqref{Linfty-symbol-b}, and \eqref{transport-symbol-scaled}, can be represented in the form 
\begin{equation}\label{kForm}
K(t, y, s, \tilde{y}) = \int e^{i\psi_{x, \xi}(t, y)} G(t, x, \xi, y) \cdot e^{-i\psi_{x, \xi}(s, \tilde y)}\tilde G(s, x, \xi, \tilde y) \, dx d\xi
\end{equation}
where the function $G$ satisfies
\[|(y - x^t)^\gamma \D_x^\alpha \D_\xi^\beta \D_y^\nu G(t, x, \xi, y)| \lesssim c_{\gamma, \alpha, \beta, \nu}
\]
and likewise for $\tilde G$.
\end{theorem}

\begin{proof}
By concatenating with $S(0, s)$, we may assume without loss of generality that $s = 0$, and write 
\[(x^s, \xi^s) = (x^0, \xi^0) = (x, \xi), \quad S(t, s) = S(t). \]
We use the FBI transform to decompose $u_0$ into coherent states, writing
\[ u(y) = (S(t, 0)T^* T u_0)(y) = \int (S(t)\phi_{x, \xi})(y) (Tu_0)(x, \xi) \, dx d\xi \]
where
\[ \phi_{x, \xi}(y) = e^{-\half(x - y)^2}e^{-i\xi(x - y)}. \]
Then we define the function $G$ by
\[ G(t, x, \xi, y) = e^{-i\psi_{x, \xi}(t, y)}(S(t)\phi_{x, \xi})(y) \]
so that $K$ has the desired form \eqref{kForm}, with
\[\tilde G(s, x, \xi, \tilde y) = e^{-\half (\tilde y - x)^2}.\]

\

It remains to prove the estimate on $G$. By the regularity estimates of Propositions~\ref{p:eikonaly} and \ref{p:eikonalx} on $\psi$, we may multiply by
\[
e^{i\psi_{x, \xi}(t, y)} e^{-i(\psi_{x_0, \xi_0}(t, y) - \xi_0(x - x_0))}
\]
for arbitrary fixed $(x_0, \xi_0)$, and prove the estimate for 
\[ G_1(t, x, \xi, y) = e^{-i\psi_{x_0, \xi_0}(t, y)}(S(t)e^{i\xi_0(x - x_0)}\phi_{x, \xi})(y) \]
at $(x, \xi) = (x_0, \xi_0)$. 

\

Next, we translate by $x_0^t$,
\[
G_2(t, x, \xi, y) =  G_1(t, x + x_0, \xi + \xi_0, y + x_0^t),
\]
so that it suffices to show
\[|y^\gamma \D_x^\alpha \D_\xi^\beta \D_y^\nu G_2(t, x, \xi, y)| \lesssim c_{\gamma, \alpha, \beta, \nu}\]
at $(x, \xi) = (0, 0)$. 

\

A direct computation shows that
\[
(i\D_t + \tilde a_2(t, y, y', D))G_2 = 0, \qquad G_2(0) = \phi_{x, \xi}
\]
where
\begin{equation*}
\begin{aligned}
\tilde a_2(t, y, y', \eta) &= - \eta \tilde a_\xi(t, x_0^t, \xi_0^t) - \tilde a\left(t, y + x_0^t, \frac{\psi_{x_0, \xi_0}(t, y + x_0^t) - \psi_{x_0, \xi_0}(t, y' + x_0^t)}{y - y'}\right)\\
&\quad + \tilde a\left(t, y + x_0^t, \eta + \frac{\psi_{x_0, \xi_0}(t, y + x_0^t) - \psi_{x_0, \xi_0}(t, y' + x_0^t)}{y - y'}\right).
\end{aligned}
\end{equation*}
Since differentiating $G_2$ in $(x, \xi)$ is given by differentiating the initial data and evolving to time $t$, and since the initial data is Schwartz, it now suffices to show that
\[
(i\D_t + \tilde a_2(t, y, y', D))v = 0, \qquad v(0) = v_0
\]
preserves Schwartz initial data $v_0$.

\

We first establish symbol properties and estimates for $\ta_2$. Observe that $\D_{y, y'} \tilde a_2 \in L^1S_{0, 0}^1$ by estimates on $\tilde a_x$ in Lemma~\ref{l:hamEst}, and the estimates on $\psi$ in Proposition~\ref{p:eikonaly}. 

The case $\D_{\eta} \tilde a_2$ is not a member of the same symbol class as $\D_{y, y'} \tilde a_2$ due to the transport term of $\ta_2$. Instead, decomposing
\begin{equation*}
\begin{aligned}
\tilde a_2(t, y, y', \eta) &= \tilde a_3(t, y, y', \eta) \\
&\quad +\eta\left(\tilde a_\xi\left(t, y + x_0^t, \frac{\psi_{x_0, \xi_0}(t, y + x_0^t) - \psi_{x_0, \xi_0}(t, y' + x_0^t)}{y - y'}\right) - \tilde a_\xi(t, x_0^t, \xi_0^t)\right) \\
&=: a_3(t, y, y', \eta) + a_4(t, y, y', \eta),
\end{aligned}
\end{equation*}
we see that $\D_{\eta} \tilde a_3 \in L^1 S_{0, 0}^1$ by using the same estimates of Lemma~\ref{l:hamEst} and Proposition~\ref{p:eikonaly}. On the other hand, 
\[
\|(\D_{\eta} \tilde a_4)(t, y, y', D)u\|_{L^1L^2} = \|(\D_{\eta} \tilde a_4)(t, y, y')u\|_{L^1L^2} \lesssim \|yu\|_{L^2}.
\]
We conclude that
\begin{equation}\label{a2bd}
\|(\D_{y, y', \eta} \tilde a_2)(t, y, y', D)u\|_{L^1L^2} \lesssim \|yu\|_{L^2} + \|u\|_{H^1}.
\end{equation}

\

To show $v$ is Schwartz, it suffices to establish energy estimates for $y^{\beta} \D_y^{\beta'} v$, which we obtain by induction on $\beta + \beta'$. The case $\beta + \beta' = 0$ is straightforward, using that the dispersive term is constant coefficient and the transport coefficient satisfies the bound $\|\tilde b_x\|_{L^1L^\infty} \lesssim 1$. 

For $\beta + \beta' = 1$, we have the following equations for $yv$ and $\D_yv$:
\begin{equation*}
\begin{aligned}
(i\D_t + \tilde a_2(t, y, y', D))(yv) &= -i(\D_\eta \tilde a_2)(t, y, y', D)v, \\
(i\D_t + \tilde a_2(t, y, y', D))(\D_y v) &= -i((\D_{y} + \D_{y'}) \tilde a_2)(t, y, y', D)v.
\end{aligned}
\end{equation*}
Using \eqref{a2bd} and Gronwall, we conclude that 
\[
\|yv(t)\|_{L^2} + \|\D_y v\|_{L^2} \lesssim \|yv_0\|_{L^2} + \|\D_y v_0\|_{L^2} + \|v_0\|_{L^2}.
\]
The cases of higher $\beta + \beta'$ then follow similarly by induction.

\end{proof}

\subsection{Dispersive estimates}

Using the representation formula of Theorem \ref{t:repformulaEik}, we prove the following dispersive estimate.

\begin{proposition}\label{p:transport-dispersive}
Let $m \in [2, 3]$, $\delta = \frac{2 - m}{2}$, and $a(t, x, \xi)$ be given by \eqref{transport-symbol} and satisfy \eqref{L1-symbol-b}, \eqref{Linfty-symbol-b}. Let $u_0$ have frequency support $\lambda$. Then the evolution operator $S(t, s)$ for $i\D_t + A^w$ satisfies the estimate
\[
\|S(t, s) u_0\|_{L^\infty_x} \lesssim \lambda^{\delta} |t - s|^{-\frac{1}{2}} \|u_0\|_{L^1}
\]
for all $t, s \in [0, 1]$.
\end{proposition}

\begin{proof}
Without loss of generality let $s = 0$. When $t < \lambda^{-m}$, the estimate is immediate from Sobolev embedding, so we may fix $\tau \in [\lambda^{-m}, 1]$ and prove the estimate when $t = \tau$. Accordingly, we apply the scaling \eqref{transport-symbol-scaled}, and set
\[
v(t, y) = (S(t, 0) u_0)(\tau t, \mu y), \qquad v_0(y) = u_0(\mu y), \qquad \mu = \tau^{\frac12} \lambda^{-\delta}.
\]
It suffices to show
\[
\|v(1)\|_{L^\infty} \lesssim \|v_0\|_{L^1}.
\]

We apply the representation formula \eqref{kForm} of Theorem~\ref{t:repformulaEik},
\begin{equation*}
\begin{aligned}
v(t, y) &= \int e^{i\psi_{x, \xi}(t, y)} G(t, x, \xi, y) \cdot e^{-i\psi_{x, \xi}(0, \tilde y)}\tilde G(0, x, \xi, \tilde y) v_0(\tilde y) \, dx d\xi d\tilde{y} \\
&= \int e^{i\psi_{x, \xi}(t, y)} G(t, x, \xi, y) \cdot e^{-i\xi(\tilde y - x)} e^{-\half(\tilde{y} - x)^2} v_0(\tilde y) \, dx d\xi d\tilde{y}.
\end{aligned}
\end{equation*}
By the frequency support of $v_0$ in $B = \{|\xi| \approx \mu\lambda \}$, the contribution of the complement of $B$ to the integral is negligible, so it suffices to consider
$$\int\int_{B} |G(t, x, \xi, y)| \, d\xi \, e^{-\half(\tilde{y} - x)^2} |v_0(\tilde{y})| \, dx d\tilde{y} \lesssim \|v_0\|_{L^1}\sup_x \int_B |G(t, x, \xi, y)| \, d\xi.$$
It remains to show
$$\int_B |G(1, x, \xi, y)| \, d\xi \lesssim 1.$$
Given the bound for $G$ in Theorem \ref{t:repformulaEik}, this reduces to showing
$$\int_B (1 + |x^1 - y|)^{-N} \, d\xi \lesssim 1.$$
Using estimate $c)$ of Proposition~\ref{p:disp}, we may change variables to obtain
$$\int_B (1 + |x^1 - y|)^{-N} \, d\xi \lesssim \int (1 + |x^1 - y|)^{-N} \, dx^1 \lesssim 1$$
as desired.

\end{proof}

For the proof of the lateral Strichartz estimate, we also prove the following lateral dispersive estimate.

\begin{proposition}\label{p:transport-dispersive-2}
Let $m \in [2, 3]$, $\delta = \frac{2 - m}{2}$, and $a(t, x, \xi)$ be given by \eqref{transport-symbol} and satisfy \eqref{L1-symbol-b}, \eqref{Linfty-symbol-b}. Let $u_0$ have frequency support $\lambda$. Then the evolution operator $S(t, s)$ for $i\D_t + A^w$ satisfies the estimate
\[
|(S(t, s) u_0)(y)| \lesssim \lambda^{\half}  |t - s|^{-\half} \int |y - \tilde y|^{-\half} |u_0(\tilde y)| \, d\tilde y
\]
for all $t, s \in [0, 1]$.
\end{proposition}

\begin{proof}
We apply the scaling \eqref{transport-symbol-scaled} with $\tau = 1$,
\[
v(t, y) = (S(t, s) u_0)(t, \mu y), \qquad v_0(y) = u_0(\mu y), \qquad \mu = \lambda^{-\delta}.
\]

Without loss of generality, we prove the estimate for $S(t, s) u_0$ at $(t, y) = (0, 0)$. Applying the representation formula \eqref{kForm} of Theorem~\ref{t:repformulaEik} for $v$, we have
\begin{equation*}
\begin{aligned}
v(0, 0) &= \int e^{i\psi_{x, \xi}(0, 0)} G(0, x, \xi, 0) \cdot e^{-i\psi_{x, \xi}(s, \tilde y)}\tilde G(s, x, \xi, \tilde y)v_0(\tilde y) \, dx d\xi d\tilde y
\end{aligned}
\end{equation*}
so that
\begin{equation*}
\begin{aligned}
|v(0, 0)| &\lesssim \int \langle x \rangle^{-N} \langle \tilde y - x^s \rangle^{-N} |v_0(\tilde y)| \, dx d\xi d\tilde y = \int \mu^{-1} \langle x \rangle^{-N} \langle \tilde \mu^{-1} y - x^s \rangle^{-N} \,  dx d\xi \cdot |u_0(\tilde y)| \, d\tilde y.
\end{aligned}
\end{equation*}

The inner integral is maximized with respect to $\tilde y$ when
\[
|\mu^{-1}\tilde y| \approx |x^s| \approx |s \cdot \dot x^s| \approx s \mu \lambda,
\]
and thus when
\[
\mu \approx \lambda^{-\half} s^{-\half} |\tilde y|^{\half}.
\]
Since we also have $\D_\xi x^s \approx s$, we estimate using a change of variables in $\xi$,
\[
\int \mu^{-1} \langle x \rangle^{-N} \langle \mu^{-1}\tilde y - x^s \rangle^{-N}  \, dx d\xi \lesssim \lambda^{\half} s^{-\half} |\tilde y|^{-\half},
\]
as desired.

\end{proof}

\subsection{Strichartz estimates} 

The proof of the Strichartz estimates \eqref{direct-strich} and \eqref{lat-strich} both use a classical $TT^*$ approach. Here we demonstrate the proof of \eqref{lat-strich}, as the former is more standard.

From Proposition~\ref{p:transport-dispersive-2}, we have
\[
|(S(t, s) u_0)(y)| \lesssim \lambda^{\half}  |t - s|^{-\half} \int |y - \tilde y|^{-\half} |u_0(\tilde y)| \, d\tilde y
\]
and thus
\[
\left\|\int_0^1 (S(t, s)F(s, \tilde y))(y) \, ds \right\|_{L_y^4 L_t^\infty} \lesssim \lambda^{\half} \left\| (|s|^{-\half} *_s |\tilde y|^{-\half} *_{\tilde y} F(s, \tilde y))(t, y) \right\|_{L_y^4 L_t^\infty}.
\]
Using Hardy-Littlewood-Sobolev twice, we have
\[
\left\|\int_0^1 (S(t, s)F(s, \tilde y))(y) \, ds \right\|_{L_y^4 L_t^\infty} \lesssim \|F\|_{L_y^{4/3} L_t^1}.
\]
Then applying H\"older's inequality,
\[
\left\|\int_0^1 (S(0, s) F(s, \tilde y)) \, ds \right\|_{L^2_y}^2 = \left|\int_0^1\int_0^1 \langle (S(t, s)F(s, \tilde y))(y), F(t, y) \rangle_y \, ds dt \right| \lesssim \lambda^{\half} \|F\|_{L^{4/3}_y L_t^1}^2,
\]
so that we obtain \eqref{direct-strich} by duality.

\section{Bilinear estimates}\label{s:bilinear}

Our objective in this section is to prove a bilinear estimate for a linear dispersive flow with transport. Similar to the Strichartz estimates, due to the variable coefficient transport term, we approach this with a physical space argument, in the form of a positive commutator argument.

Here, we denote the control parameters
\begin{equation}\label{control-parameters}
\AA(t) = \| \langle D \rangle^{\frac14 (1 - 3\alpha)} \phi(t)\|_{L_x^\infty}, \qquad \BB(t) = \| \langle D \rangle^{\half (1 - \alpha)} \phi(t)\|_{L_x^\infty}.
\end{equation}

\begin{proposition}\label{p:bilinear}
For smooth solutions $u_\lambda, v_\mu$ to
\begin{equation*}
\begin{aligned}
(\D_t -|D|^\alpha \D_x - \phi_{< \lambda'} \D_x) u_\lambda &= f_1, \\
(\D_t -|D|^\alpha \D_x - \phi_{< \mu'} \D_x) v_\mu &= f_2
\end{aligned}
\end{equation*}
on $I = [0, T]$, with frequency supports $\lambda \gg_\AA \mu \geq 1$ respectively, as well as low-frequency truncations $\lambda', \mu' \leq 1$, we have
\begin{equation}
\begin{aligned}
\|u_\lambda v_\mu\|_{L^2(I;L^2)} \lesssim_{\|\AA\|_{L^\infty_t}} \lambda^{-\frac{\alpha}{2}}(&\|u_\lambda\|_{L^\infty(I;L^2)}\|v_\mu\|_{L^\infty(I;L^2)} \\
&\quad + \|f_1\|_{L^1(I;L^2)}\|v_\mu\|_{L^\infty(I;L^2)}  + \|f_2\|_{L^1(I;L^2)}\|u_\lambda\|_{L^\infty(I;L^2)}).
\end{aligned}
\end{equation}
\end{proposition}

\begin{proof}
Rescaling
\[ u_\lambda \mapsto u_\lambda(\mu^{-1 - \alpha} t, \mu^{-1} x), \quad v_\mu \mapsto v_\mu(\mu^{-1 - \alpha} t, \mu^{-1} x), \]
\[ \phi_{< \lambda'} \mapsto \mu^{-\alpha} \phi_{< \lambda'}(\mu^{-1 - \alpha} t, \mu^{-1} x), \quad \phi_{< \mu'} \mapsto \mu^{-\alpha} \phi_{< \mu'}(\mu^{-1 - \alpha} t, \mu^{-1} x), \]
we may assume $\mu = 1$. Further, we write $v = v_\mu$ and $u = u_\lambda$ for brevity. Then $v$ has frequency support  $\approx 1$ so that $v = P_0 v$. Then denoting the kernel of $P_0$ by $m_0$, and fixing Schwartz $\chi \in C^\infty(\R)$ with frequency support in $[-1, 1]$ and $\chi \geq 1$ on $[-1, 1]$, we have 
\begin{equation}\label{uvproduct}
\begin{aligned}
\|u v\|_{L^2_x}^2 &= \int |u(x_1)|^2 \left|\int m_0(x_1 - x_2) v(x_2)\, dx_2 \right|^2 \, dx_1 \\
&\lesssim \int |m_0(x_1 - x_2)|^2 \cdot |u(x_1)v(x_2)|^2\, dx_2  \, dx_1 \\
&\lesssim \int \chi^2(x_1 - x_2) |u(x_1) v(x_2)|^2\, dx_2  \, dx_1.
\end{aligned}
\end{equation}
Defining $U(x_1, x_2) = u(x_1) v(x_2)$, it thus suffices to bound the diagonal-weighted $L^2$ norm of $U$ on $\R^2$,
\[\| \chi(x_1 - x_2)U\|_{L^2_{x}}^2.\]

\

Define $\phi = (\phi_{< \lambda'}(x_1), \phi_{< \mu'}(x_2))$, $f = (f_1(x_1), f_2(x_2))$, and $\nabla = (\D_{x_1}, \D_{x_2})$. Then $U$ satisfies
\[ (\D_t - |D_{x_1}|^\alpha \D_{x_1} -  |D_{x_2}|^\alpha \D_{x_2} - \phi \cdot \nabla) U = f \cdot (v, u). \]
Let $w' = \chi^2$. Then we have, denoting $w = w(x_1 - x_2)$,
\begin{equation}\label{positivederivative}
\begin{aligned}
\frac{d}{dt} \langle w(x_1 - x_2) U, U\rangle &= \langle w U_t, U \rangle + \langle w U, U_t \rangle\\
&= \langle w\cdot (|D_{x_1}|^\alpha \D_{x_1} + |D_{x_2}|^\alpha \D_{x_2} + \phi\cdot \nabla) U, U \rangle + \langle wf \cdot (v, u), U \rangle\\
&\quad+ \langle w U, (|D_{x_1}|^\alpha \D_{x_1} + |D_{x_2}|^\alpha \D_{x_2} + \phi \cdot \nabla)U \rangle + \langle wU, f \cdot (v, u)\rangle\\
&= \langle [w, (|D_{x_1}|^\alpha \D_{x_1} + |D_{x_2}|^\alpha \D_{x_2})] U, U \rangle - \langle (\nabla \cdot (w\phi)) U, U \rangle \\
&\quad + 2\Re \langle wf \cdot (v, u), U \rangle \\
&= I + II + 2\Re \langle wf \cdot (v, u), U \rangle.
\end{aligned}
\end{equation}

Consider first $I$ on the right hand side of (\ref{positivederivative}). Set $p(D_{x_1}) = |D_{x_1}|^\alpha \D_{x_1}$ and write
\begin{equation*}
\begin{aligned}
{}[w(x_1 - x_2), p(D_{x_1})] &= i w'(x_1 - x_2) p'(D_{x_1}) + R(x_1, y_1, D)
\end{aligned}
\end{equation*}
where
\[ R(x_1, y_1, \xi) = p''(\xi_1) \int_0^1 w''(hx_1 + (1 - h)y_1 - x_2)h \, dh. \]
Using the frequency localization of $U$ at $\lambda$, we bound the contribution of $R$ by
\[|\langle RU, U \rangle| \lesssim \lambda^{\alpha - 1}|\langle \chi^2(x_1 - x_2) U, U \rangle| \ll \lambda^{\alpha}|\langle \chi^2(x_1 - x_2) U, U \rangle|. \]
Similarly, we may exchange the principal order term by its symmetrization. Precisely, we have $w' = \chi^2$ and write 
\begin{equation*}
\begin{aligned}
\chi^2(x_1 - x_2)p'(D_{x_1}) U &= \chi(x_1 - x_2)p'(D_{x_1})\chi(x_1 - x_2) U \\
&\quad + \chi(x_1 - x_2)[\chi(x_1 - x_2), p'(D_{x_1})] U
\end{aligned}
\end{equation*}
and bound the commutator error by
\[| \langle \chi(x_1 - x_2)[\chi(x_1 - x_2), p'(D_{x_1})] U , U \rangle| \ll \lambda^{\alpha}|\langle \chi^2(x_1 - x_2) U, U \rangle|. \]
The symmetrized principal term itself then satisfies
\[ \langle ip'(D_{x_1})\chi(x_1 - x_2) U , \chi(x_1 - x_2) U \rangle \approx \lambda^{\alpha}\| \chi(x_1 - x_2) U \|_{L^2}. \]

Similar estimates hold for the case $p(D_{x_2}) = |D_{x_2}|^\alpha \D_{x_2}$, though with opposite sign since $w = w(x_1 - x_2)$, and with principal term satisfying instead
\[ \langle ip'(D_{x_2})\chi(x_1 - x_2) U , \chi(x_1 - x_2) U \rangle \approx \| \chi(x_1 - x_2) U, U \|_{L^2} \ll \lambda^{\alpha}\| \chi(x_1 - x_2) U \|_{L^2}. \]
We conclude 
\[
I = (1 + o(1))\lambda^{\alpha} \| \chi(x_1 - x_2) U \|_{L^2}^2.
\]

\

Next, we bound $II$ on the right hand side of \eqref{positivederivative}. In the case where the derivative falls on $\phi$, we have
\[ |\langle w(\nabla \cdot \phi) U, U \rangle| \lesssim \|\D_x \phi\|_{L^\infty} \|U\|_{L^2}^2 \lesssim_\AA \|U\|_{L^2}^2. \]
In the case where the derivative falls on $w$, we have
\begin{equation*}
\begin{aligned}
| \langle w'(x_1 - x_2) (\phi_{< \lambda'}(x_1) - \phi_{< \mu'}(x_2)) U,  U \rangle | &\lesssim \| \phi_{\leq 0}\|_{L^\infty} |\langle \chi^2(x_1 - x_2) U, U \rangle| \\
& \lesssim_\AA \| \chi(x_1 - x_2) U \|_{L^2}^2.
\end{aligned}
\end{equation*}

\

We conclude from \eqref{positivederivative} and the analysis of $I$ and $II$ that
\begin{equation}\label{positivederivative2}
\begin{aligned}
\Big|\frac{d}{dt} \langle w(x_1 - x_2) U, U\rangle &- (1 + o(1))\lambda^{\alpha} \| \chi(x_1 - x_2) U \|_{L^2}^2\Big| \\
&\quad \lesssim_\AA \|U\|_{L^2}^2 + \|f \cdot (v, u)\|_{L^2} \|U\|_{L^2}.
\end{aligned}
\end{equation}
Integrating \eqref{positivederivative2} in $t$, we have
\begin{equation*}
\begin{aligned}
\lambda^{\alpha} \| \chi(x_1 - x_2) U \|_{L^2_{t, x}}^2 &\lesssim_{\|\AA\|_{L^\infty_t}} \|U\|_{L^\infty_t L^2_x}^2 + \|U(0)\|_{L^2}^2 + \|U(T)\|_{L^2}^2 + \|f \cdot (v, u)\|_{L^1_tL^2_x}  \|U\|_{L^\infty_t L^2_x}. 
\end{aligned}
\end{equation*}
Using \eqref{uvproduct}, we have 
\begin{equation*}
\begin{aligned}
  \| uv\|_{L^2_{t, x}}^2 \lesssim_{\|\AA\|_{L^\infty_t}} \lambda^{-\alpha}(&\|u\|_{L^\infty_t L^2_x}^2 \|v\|_{L^\infty_t L^2_x}^2 + \|u(0)\|_{L^2_x}^2\|v(0)\|_{L^2_x}^2 + \|u(T)\|_{L^2_x}^2\|v(T)\|_{L^2_x}^2 \\
  &\quad + ( \|f_1\|_{L^1_t L^2_x}\|v\|_{L^\infty_t L^2_x} + \|f_2\|_{L^1_t L^2_x}\|u\|_{L^\infty_t L^2_x})\|u\|_{L^\infty_t L^2_x} \|v\|_{L^\infty_t L^2_x}). 
\end{aligned}
\end{equation*}
which establishes the estimate of the proposition.

\end{proof}

\section{Normal form analysis}\label{s:nf}

Our objective in this section is to perform a normal form analysis to address the quadratic nonlinearity $\half \D_x(\phi^2)$ on the right hand side of \eqref{BO}. The analysis proceeds by decomposing this term into three components and treating each separately as follows.

\

To begin, we collect the paradifferential components of the nonlinear term which are immediately amenable to a normal form correction. Precisely, we define the bilinear form
\begin{equation}\label{Q2k-def}
Q^2_k(u, v) := P_k^+(u_{\geq k} \D_x v_{< k} ) + \half \D_x P_k^+(u_{\geq k} v_{\geq k} ) + [P_k^+, u_{< k}] \D_x v,
\end{equation}
where $P_k^+$ denotes the Littlewood-Paley projection further restricted to positive frequencies $[0, \infty)$. Observe that the quadratic terms in $Q^2_k$ are of order 0, in the sense that the high frequency variables are either undifferentiated or possess a commutator structure. In particular, we will see that they are directly amenable to removal by a normal form transformation.

Then we can rewrite \eqref{BO} as
\begin{align}\label{BO2}
(\D_t - |D|^\alpha \D_x)\phi_k^+ - \phi_{< k} \D_x \phi_k^+ &= Q_k^2(\phi, \phi).
\end{align}
In contrast to $Q^2_k$, the paradifferential quadratic term $\phi_{< k} \D_x \phi_k^+$ on the left hand side of \eqref{BO2} is of order 1, which is too unbalanced for a classical normal form to be effective. Instead, we first address this term using a pseudodifferential exponential conjugation, which we define later in this section. This suffices to remove most of the term $\phi_{< k} \D_x \phi_k^+$, but leaves two categories of residual non-perturbative terms: 
\begin{enumerate}[1)]
\item Order 0 and order $1 - \alpha$ quadratic terms. These, along with $Q_k^2$ discussed above, may be removed directly via paradifferential normal forms, now that they are lower order.
\item An order 1 term with a very low frequency coefficient,
\[
\phi_{\leq k'} \D_x \phi_k^+.
\]
Here, we denote $k' = \frac12(1 - \alpha)k$, where $2^{-k'} \gg 1$ is the unit-time spatial scale for a wave packet at frequency $2^k \gg 1$ under a dispersive flow of order $1 + \alpha$.

Unlike the special case of the Benjamin-Ono equation where the corresponding exponential conjugation is a simple product, the corresponding exponential conjugation here would be pseudodifferential and in particular unbounded. Instead, we may address this component by viewing it as part of the linear analysis, and using variable coefficient linear Strichartz and bilinear estimates established in Sections~\ref{s:transport-strich} and \ref{s:bilinear}.
\end{enumerate}

\subsection{Multi-linear notation}

We will use a convenient notation for describing multi-linear expressions of product type (see \cite{Ltao, ifrim2017well}). We let $L(\phi_1, ..., \phi_n)$ denote a translation-invariant expression of the form
\[
L(\phi_1, ..., \phi_n)(x) = \int K(y) \phi_1(x + y_1)... \phi_n(x + y_n) \, dy, \qquad K \in L^1.
\]
By $L_k$, we denote such expressions whose output is localized at frequency $2^k$.

The $L$ notation has several convenient features. For instance, it behaves well with respect to iteration,
\[
L(L(u, v), w) = L(u, v, w).
\]
We also have the usual H\"older's inequality,
\[
\|L(u, v) \|_{L^r} \lesssim \|u\|_{L^p} \|v\|_{L^q}, \qquad \frac{1}{p} + \frac{1}{q} = \frac{1}{r}.
\]
We can also apply bilinear estimates to $L$ forms. To see this, we need to account for the uncorrelated translations which appear within the $L$ form. We use the translation group $\{T_y\}_{y \in \R}$, 
\[
(T_y u)(x) = u(x + y),
\]
and estimate
\[
\|L(u, v) \|_{L^r} \lesssim \sup_{y}\|uT_y v\|_{L^r}.
\]
As a result, it will be necessary to state our bilinear estimates in the mildly generalized setting of these uncorrelated translations.

\subsection{Exponential conjugation} Define $\Phi(t, x)$ by
\begin{equation}\label{Phi-def}
\Phi_x = \phi, \qquad \Phi(0, 0) = 0,
\end{equation}
so that $\Phi$ solves the equation
\[ (\D_t - |D|^\alpha \D_x) \Phi = \half \Phi_x^2. \]

Then define the symbol
\begin{equation}\label{a-def}
a_{(k', k)}(t, y, \xi) = (1 + \alpha)^{-1} \Phi_{(k', k)}(y)\xi|\xi|^{-\alpha}
\end{equation}
and denote the operator with symbol $e^{ia_{(k', k)}(t, y, \xi)}$ by
\begin{equation}\label{a-def-op}
e^{iA_{(k', k)}} (t, y, D) = Op(e^{ia_{(k', k)} (t, y, \xi)}).
\end{equation}
Here, we fix a quantization with multiplication on the right, although this is inessential. We will often abbreviate \eqref{a-def} and \eqref{a-def-op} respectively as
\[
a = a_{(k', k)}(t, y, \xi), \qquad e^{iA} = e^{iA_{(k', k)}}(t, y, D).
\]

We will establish the $L^p$-boundedness of $e^{iA}$ in Section~\ref{p:expBd}. 

\

We define the exponentially conjugated variable,
\begin{equation}\label{psi-def}
\psi_k^+ := P^+_k e^{iA} \phi_k^+
\end{equation}
which satisfies 
\begin{equation}\label{expError}
\begin{aligned}
(\D_t - |D|^\alpha \D_x &- \phi_{\leq k'} \D_x) \psi_k^+ \\
&= P^+_k (e^{iA} (Q_k^2(\phi, \phi) + \phi_{(k', k)} \D_x \phi_k^+) + [(\D_t - |D|^\alpha \D_x) - \phi_{\leq k'} \D_x, e^{iA}] \phi_k^+).
\end{aligned}
\end{equation}

We observe in the following lemma that the principal term of the $\D_t - |D|^\alpha \D_x$ commutator cancels the first order quadratic term $\phi_{(k', k)} \D_x \phi_k^+$. The higher order terms produce source terms which are cubic or higher, or quadratic but better balanced:
\begin{lemma}\label{l:expCom}
For $f = P_kf$, we have
\begin{equation}\label{Dtcomm}
\begin{aligned}
[][(\D_t - |D|^\alpha \D_x), e^{iA}]f &= -e^{iA} (\phi_{(k', k)} \D_x f ) \\
&\quad - (1 + \alpha)^{-1}e^{iA} [(\D_x \phi_{(k', k)})f] \\
&\quad + (1 + \alpha)^{-1} |D|^{-\alpha} e^{iA}\D_x [(|D|^\alpha \phi_{(k', k)})f] \\
&\quad -\frac{\alpha(\alpha - 1)}{2(1 + \alpha)} \D_x |D|^{-\alpha}e^{iA}\int_0^1(1 - h) \\
&\quad \cdot\int e^{i(x - y)\xi} |\xi|^{\alpha - 2} (\D_x \phi_{(k', k)} )(y + h(x - y)) \D_xf(y) \, dy d\xi \, dh \\
&\quad + Q_k^{exp, 3+}(f)
\end{aligned}
\end{equation}
where $Q_k^{exp, 3+}$ consists of perturbative cubic terms,
\begin{equation*}
\begin{aligned}
&Q_k^{exp, 3+}(f) \\
&= (1 + \alpha)^{-2} |D|^{-2\alpha}\D_x[e^{iA} (\phi_{(k', k)} [|D|^\alpha \phi_{(k', k)} + \half P_{(k', k)} (\phi^2)] f)]  \\
&\quad + \half(1 + \alpha)^{-1}|D|^{-\alpha}  e^{iA} \D_x [P_{(k', k)} (\phi^2) f] -(1 + \alpha)^{-2} |D|^{-\alpha}\D_x[ e^{iA} ( \phi_{(k', k)}^2 f)] \\
&\quad+ \frac{\alpha (\alpha - 1)}{2(1 + \alpha)^2}|D|^{2 - 2\alpha}\int_0^1\int_0^1\int e^{i(x - z)\eta} e^{i(z - y)\xi} \D_xf(y)\\
&\quad \cdot \Big[(1 - h)e^{ia(y + h(z - y), \eta)}|\xi|^{\alpha - 2} \phi_{(k', k)}^2(y + h(z - y)) \\
&\quad\phantom{\cdot \Big[ } -i (\alpha - 2)(1 - h)^2 e^{ia(z + h'(1 - h) (y - z), \eta)}  |\xi|^{\alpha - 4}\xi \cdot \phi_{(k', k)}(z + h'(1 - h) (y - z)) \\
&\quad\phantom{\cdot \Big[ + (1 - h)^2 e^{ia(z + h'(1 - h) (y - z), \eta)} (y - z)} \cdot \D_x \phi_{(k', k)} (y + h(z - y))\Big]  \, dy d\xi dz d\eta \, dh' dh.
\end{aligned}
\end{equation*}
\end{lemma}

\begin{proof}
We divide the commutator into three components and handle each separately:
\begin{equation}\label{expComm}
\begin{aligned}
{}[\D_t - |D|^\alpha \D_x, e^{iA}]f &= (\D_{t} e^{iA})f - |D|^\alpha (\D_x e^{iA})f - [|D|^\alpha, e^{iA}] \D_x f.
\end{aligned}
\end{equation}

\

We begin with the first term on the right hand side of \eqref{expComm}. We have
\[(\D_{t} e^{iA}) f = (1 + \alpha)^{-1} |D|^{-\alpha} \D_x [e^{iA} ([\D_{t} \phi_{(k', k)}] f)]. \]
Applying the Leibniz rule and using the equation to replace time derivatives, we find
\begin{align*}
\D_x [e^{iA} ([\D_{t} \phi_{(k', k)}] f)] &= (1 + \alpha)^{-1} |D|^{-\alpha}\D_x [e^{iA} (\phi_{(k', k)} [|D|^\alpha \phi_{(k', k)} + \half P_{(k', k)} (\phi^2)] f)] \\
&\quad + e^{iA} \D_x ([|D|^\alpha \phi_{(k', k)} + \half P_{(k', k)} (\phi^2)] f).
\end{align*}
Observe that all the resulting terms are cubic or higher order perturbative terms, except for a single quadratic term on the second line. We see that these terms contribute the third quadratic term on the right hand side of \eqref{Dtcomm} and several perturbative terms in $Q_k^{exp, 3+}(f)$.

\

Next, we compute the second term on the right hand side of \eqref{expComm}, applying the Leibniz rule:
\begin{align*}
- |D|^\alpha (\D_x e^{iA})f &= - (1 + \alpha)^{-1} \D_x [e^{iA} ( \phi_{(k', k)} f)] \\
&= -(1 + \alpha)^{-2} |D|^{-\alpha}\D_x[ e^{iA} ( \phi_{(k', k)}^2 f)] - (1 + \alpha)^{-1} e^{iA} \D_x( \phi_{(k', k)} f) \\
&= -(1 + \alpha)^{-2} |D|^{-\alpha}\D_x[ e^{iA} ( \phi_{(k', k)}^2 f)] \\
&\quad - (1 + \alpha)^{-1} e^{iA} ( (\D_x \phi_{(k', k)}) f) - (1 + \alpha)^{-1} e^{iA} ( \phi_{(k', k)} \D_x f).
\end{align*}
The last term contributes part (in the sense that the constant in front is only $(1 + \alpha)^{-1} < 1$) of the first quadratic term on the right hand side of \eqref{Dtcomm}. The other terms contribute the second quadratic term on the right hand side of \eqref{Dtcomm}, and one term in $Q_k^{exp, 3+}(f)$.

\

Lastly, we rewrite the third term on the right hand side of \eqref{expComm}. Precisely, we use the following identity, expanding the commutator into its principal symbol with second order remainder: 
\begin{align*}
[&P(y, D), Q(D)]g(x) = i\D_\eta Q(D) (\D_y P)(y, D)]g(x) \\
& -\half \int e^{i(x - z)\eta}e^{i(z - y)\xi} \left[\int_0^1 (1 - h)(\D_y^2 P)(y + h(z - y), \eta) \, dh \right]  \D_\xi^2 Q(\xi) g(y) \, dy d\xi dz d\eta.
\end{align*}

In our setting, the principal term of the commutator has the form
\[\alpha |D|^{\alpha - 1} [(\D_x e^{iA}) \D_x f] = - \alpha (1 + \alpha)^{-1} e^{iA}(\phi_{(k', k)} \D_x f), \]
which is precisely the remaining part of the first quadratic term on the right hand side of \eqref{Dtcomm}. Then using the remainder identity, we have 
\begin{equation*}
\begin{aligned}
-& [|D|^\alpha, e^{iA}] \D_xf + \alpha (1 + \alpha)^{-1} e^{iA}(\phi_{(k', k)} \D_x f) \\
&= -\half \int e^{i(x - z)\eta}e^{i(z - y)\xi} \left[\int_0^1 (1 - h) (\D_x^2 e^{ia})(y + h(z - y), \eta) \, dh \right]  (\D_\xi^2 |\xi|^{\alpha}) \D_xf(y) \, dy d\xi dz d\eta \\
&= -\half \int e^{i(x - z)\eta} e^{i(z - y)\xi} \\
&\quad \phantom{-\half \int}\cdot\int_0^1(1 - h) e^{ia(y + h(z - y), \eta)}[-(1 + \alpha)^{-2}|\eta|^{2 -2\alpha} \phi_{(k', k)}^2(y + h(z - y)) \\
&\quad \phantom{-\half \int\int_0^1(1 - h) e^{iA(y + h(z - y), \eta)}(}+ i(1 + \alpha)^{-1}\eta |\eta|^{-\alpha} \D_x \phi_{(k', k)} (y + h(z - y))] \, dh \\
&\quad \phantom{-\half\int}\cdot \alpha (\alpha - 1)|\xi|^{\alpha - 2} \D_xf(y) \, dy d\xi dz d\eta.
\end{aligned}
\end{equation*}
On the right, we have a cubic perturbative term, but also a quadratic term which will later require a normal form correction. In preparation, we commute the exponential to the front, using the identity
\begin{align*}
e^{ia(y + h(z - y), \eta)} &= e^{ia(z, \eta)} + (1 + \alpha)^{-1}\int_{0}^{1}(1 - h) (y - z) i\eta |\eta|^{-\alpha} \phi_{(k', k)}(z + h'(1 - h) (y - z)) \\
&\quad \phantom{ e^{ia(z, \eta)} + (1 + \alpha)^{-1}\int_{0}^{1} }\cdot e^{ia(z + h'(1 - h) (y - z), \eta)} \, dh'
\end{align*}
to obtain and write
\begin{equation*}
\begin{aligned}
-& [|D|^\alpha, e^{iA}] \D_xf + \alpha (1 + \alpha)^{-1} e^{iA}(\phi_{(k', k)} \D_x f) \\
&= \frac{\alpha (\alpha - 1)}{2(1 + \alpha)^2}\int_0^1\int_0^1\int e^{i(x - z)\eta} e^{i(z - y)\xi}|\xi|^{\alpha - 2} \D_xf(y)|\eta|^{2 -2\alpha} \\
&\quad \cdot \Big[(1 - h)e^{ia(y + h(z - y), \eta)} \phi_{(k', k)}^2(y + h(z - y)) \\
&\quad\phantom{\cdot \Big[ } + (1 - h)^2 e^{ia(z + h'(1 - h) (y - z), \eta)} (y - z)\phi_{(k', k)}(z + h'(1 - h) (y - z)) \\
&\quad\phantom{\cdot \Big[ + (1 - h)^2 e^{ia(z + h'(1 - h) (y - z), \eta)} (y - z)} \cdot \D_x \phi_{(k', k)} (y + h(z - y))\Big]  \, dy d\xi dz d\eta \, dh' dh\\
&\quad -\frac{\alpha (\alpha - 1)}{2(1 + \alpha)} \D_x |D|^{-\alpha}e^{iA}\int_0^1(1 - h)\int e^{i(x - y)\xi} |\xi|^{\alpha - 2}( \D_x \phi_{(k', k)}) (y + h(x - y)) \D_xf(y) \, dy d\xi  \, dh.
\end{aligned}
\end{equation*}
Integrating by parts in the second term on the right with respect to $\xi$, we obtain the remaining two perturbative cubic terms in $Q_k^{exp, 3+}(f)$, followed by the remaining quadratic term.
\end{proof}

\subsection{Bounds on the exponential conjugation} We consider the $L^p$-boundedness properties of the exponential conjugation $e^{iA}$ given by \eqref{a-def-op}. Recall that we denote the control parameters \eqref{control-parameters} by
\begin{equation*}
\AA(t) = \| \langle D \rangle^{\frac14 (1 - 3\alpha)} \phi(t)\|_{L_x^\infty}, \qquad \BB(t) = \| \langle D \rangle^{\half (1 - \alpha)} \phi(t)\|_{L_x^\infty}.
\end{equation*}

\begin{proposition}\label{p:expBd}
Let $f = P_kf$. We may estimate
\[
|e^{iA}f| \lesssim_\AA L(|f|).
\]
In particular, we have the $L^p$ bounds for $p \in [1, \infty]$,
\[ \|e^{iA} f \|_{L^p} \approx_\AA \|f\|_{L^p},\]
and may also estimate
\[
|e^{iA}f \cdot g| \lesssim_\AA L(|f|, |g|).
\]

\end{proposition}

\begin{proof}

We write
\[
K(x, y) = \int e^{i(x - y)\xi} e^{ia} \, d\xi
\]
so that
\[
e^{iA} f = \int K(x, y) f(y) \, dy.
\]
Since $f$ is dyadically localized at frequency $2^k$, we have the crude bound
\begin{equation}\label{K0}
|K| \lesssim 2^k.
\end{equation}

For decay away from the diagonal, we integrate by parts,
\begin{equation*}\begin{aligned}
(x - y) \cdot K = \int e^{i(x - y)\xi} \cdot i\D_\xi e^{ia} \, d\xi = -\int e^{i(x - y)\xi} \cdot \D_\xi a \cdot e^{ia} \, d\xi
\end{aligned}
\end{equation*}
where
\[
\D_\xi a_{(k', k)}(t, y, \xi) = \frac{1 - \alpha}{1 + \alpha} \Phi_{(k', k)}(y)|\xi|^{-\alpha}.
\]

To estimate $\Phi_{(k', k)}$, we have
\[
\|\Phi_{(k', k)} \|_{L^\infty} \lesssim \| \Phi_{(k', 0)} \|_{L^\infty} + \sum_{0 \leq j < k} \|\Phi_{j} \|_{L^\infty} \lesssim_\AA 2^{-k'} + \sum_{0 \leq j < k} 2^{-j} 2^{-\frac14 (1 - 3\alpha)j} \lesssim 2^{-k'} + 2^{\frac14(3\alpha - 5)k}.
\]
Since $\alpha \in [1, 2]$ and in particular $\alpha < 3$, we have $\frac14(3\alpha - 5)k < \half (\alpha - 1) k = -k'$ and thus
\[
\|\Phi_{(k', k)} \|_{L^\infty} \lesssim_\AA 2^{-k'}.
\]

As a result, for $|\xi| \approx 2^k$, we conclude
\[
|\D_\xi a| \lesssim_\AA 2^{-k'} 2^{-\alpha k} = 2^{-k},
\]
and thus 
\[
|x - y| \cdot |K| \lesssim_\AA 1.
\]

We integrate by parts once more, to write
\begin{equation*}\begin{aligned}
(x - y)^2 \cdot K = \int e^{i(x - y)\xi} \cdot (i\D_\xi)^2 e^{ia} \, d\xi = -i\int e^{i(x - y)\xi} \cdot (\D_\xi^2 a \cdot e^{ia} + i(\D_\xi a)^2 \cdot e^{ia}) \, d\xi.
\end{aligned}
\end{equation*}
Since
\[
\D_\xi^2 a_{(k', k)}(t, y, \xi) = \frac{-\alpha(1 - \alpha)}{1 + \alpha} \Phi_{(k', k)}(y)\xi|\xi|^{-\alpha - 2},
\]
we have
\[
|\D_\xi^2 a + i (\D_\xi a)^2| \lesssim_\AA 2^{-2k}
\]
so that
\[
|x - y|^2 \cdot |K| \lesssim_\AA 2^{-k}.
\]
Using this estimate for $|x - y| \geq 2^{-k}$ and \eqref{K0} for $|x - y|  \leq 2^{-k}$, we conclude
\[
|K| \lesssim_\AA \langle x - y \rangle^{-2}.
\]
Thus,
\[
|e^{iA} f| \leq \int |K(x, y)| \cdot |f(y)| \, dy \lesssim_\AA \int \langle x - y \rangle^{-2} \cdot |f(y)| \, dy = L(|f|)
\]
as desired.

\end{proof}

\subsection{Paradifferential normal form transformation} 

We recall the order 0 quadratic terms collected in \eqref{Q2k-def}:
\[
Q^2_k(u, v) = P_k^+(u_{\geq k} \D_x v_{< k} ) + \half \D_x P_k^+(u_{\geq k} v_{\geq k} ) + [P_k^+, u_{< k}] \D_x v.
\] 
In addition to $Q^2_k$, we define quadratic forms corresponding to the residual quadratic terms \eqref{Dtcomm} which appear after applying the exponential conjugation $e^{iA}$. These were computed in the previous section in Lemma~\ref{l:expCom}:
\begin{equation}\label{Q-def}
\begin{aligned}
Q_{k, I}^{exp, 2}(u, v) &:= (\D_x u_{(k', k)})P_k^+v, \\
Q_{k, II}^{exp, 2}(u, v) &:= (|D|^\alpha u_{(k', k)})P_k^+v, \\
Q_{k, h}^{exp, 2}(u, v) &:= \int e^{i(x - y)\xi} |\xi|^{\alpha - 2} (\D_x u_{(k', k)} )(y + h(x - y)) \D_xP_k^+v(y) \, dy d\xi.
\end{aligned}
\end{equation}
Lastly, for later use in the analysis of the linearized equation, we consider an additional quadratic form,
\[
Q_k^{lin, 2}(v, \phi) := v_{(0, k)} \D_x \phi_k^+.
\]
Our objective in this section is to construct normal form corrections associated to these quadratic forms.

\

Denote the dispersion relation of \eqref{BO} by 
\[
\omega(\xi) = -\xi |\xi|^\alpha
\]
and define the resonance function
\[
\Omega(\xi_1, \xi_2) = \omega(\xi_1) + \omega(\xi_2) - \omega(\xi_1 + \xi_2).
\]
Then we define bilinear normal form corrections of the form
\begin{equation}\label{Bk-def}
\begin{aligned}
\hat B_k^2(u, v)(\xi) &:= \int_{\xi_1 + \xi_2 = \xi} \Omega^{-1}(\xi_1, \xi_2)\hat Q_k^2(\xi_1, \xi_2)\hat u(\xi_1) \hat v(\xi_2) \, d\xi_1, \\
\hat B_{k, \cdot}^{exp, 2}(u, v)(\xi) &:= \int_{\xi_1 + \xi_2 = \xi} \Omega^{-1}(\xi_1, \xi_2)\hat Q_{k, \cdot}^{exp, 2}(\xi_1, \xi_2)\hat u(\xi_1) \hat v(\xi_2) \, d\xi_1, \\
\hat B_k^{lin, 2}(u, v)(\xi) &:= \int_{\xi_1 + \xi_2 = \xi} \Omega^{-1}(\xi_1, \xi_2)\hat Q_k^{lin, 2}(\xi_1, \xi_2)\hat u(\xi_1) \hat v(\xi_2) \, d\xi_1,
\end{aligned}
\end{equation}
where, writing $\xi = \xi_1 + \xi_2$,
\begin{equation}\label{Qk-def}
\begin{aligned}
\hat Q_k^{2}(\xi_1, \xi_2) &= \hat{P}_k^+(\xi)\hat P_{\geq k}(\xi_1) i\xi_2 \hat P_{< k}(\xi_2) + \half i\xi \hat{P}_k^+(\xi) \hat P_{\geq k}(\xi_1)\hat P_{\geq k}(\xi_2) \\
&\quad + [\hat P_k^+(\xi) - \hat P_k^+(\xi_2)] \hat P_{< k}(\xi_1) i\xi_2,\\
\hat Q_{k, I}^{exp, 2}(\xi_1, \xi_2) &= - i\xi_1 \hat{P}_{(k', k)}(\xi_1)\hat{P}_k^+(\xi_2), \\
\hat Q_{k, II}^{exp, 2}(\xi_1, \xi_2) &= |\xi_1|^{\alpha}\hat{P}_{(k', k)}(\xi_1)\hat{P}_k^+(\xi_2), \\
\hat Q_{k, h}^{exp, 2}(\xi_1, \xi_2) &= |(1 - h)\xi_1 + \xi_2|^{\alpha - 2} \xi_1 P_{(k', k)}(\xi_1) \xi_2 P_k^+(\xi_2), \\
\hat Q_k^{lin, 2}(\xi_1, \xi_2) &= - i\xi_2 \hat{P}_{(0, k)}(\xi_1)\hat{P}_k^+(\xi_2).
\end{aligned}
\end{equation}

Observe that $\Omega(\xi_1, \xi_2)$ vanishes along the three lines $\xi_1 = 0$, $\xi_2 = 0$, and $\xi_1 + \xi_2 = 0$, but the terms of $Q_k^2$, $Q_{k, \cdot}^{exp, 2}$, and $Q_k^{lin, 2}$ are either supported away from low frequencies, differentiated at low-frequency, or in commutator form. 

To establish precise estimates for $B_k^2$, $B_{k, \cdot}^{exp, 2}$, and $B^{lin, 2}_k$, we recall the following approximate identity for the resonance function $\Omega$, which may be established by a Taylor expansion (see for instance \cite{herr}):

\begin{lemma}\label{l:approxOmega}
For $\alpha > 0$, we have 
\[
|\Omega(\xi_1, \xi_2)| \approx |\xi_{min}||\xi_{max}|^{\alpha},
\]
where $|\xi_{min}|=\min\{ |\xi_1|,|\xi_2|,|\xi_1+\xi_2|\} $
and $|\xi_{max}|=\max\{ |\xi_1|,|\xi_2|,|\xi_1+\xi_2|\} $.
\end{lemma} 
Using this approximation, we have the following boundedness of the normal form corrections $B_k$. 

\begin{proposition}\label{p:Bk-est}
Let $\gamma_1 + \gamma_2 + \gamma_3 = 1$ and $\gamma_i \geq 0$. The bilinear forms $B_k^2$, $B_{k, \cdot}^{exp, 2}$, and $B_k^{lin, 2}$ defined in \eqref{Bk-def} may be expressed in the form
\begin{equation*}
\begin{aligned}
 B_k^2(u, v) &= |D|^{-\alpha \gamma_3} L_k(|D|^{-\alpha \gamma_1} u, |D|^{-\alpha \gamma_2} v), \\
 B_{k, I}^{exp, 2}(u, v) &= |D|^{-\alpha \gamma_3} L_k(|D|^{-\alpha \gamma_1} u, |D|^{-\alpha \gamma_2} v), \\
 B_{k, II}^{exp, 2}(u, v) &= |D|^{-\alpha \gamma_3} L_k(|D|^{(\alpha - 1) -\alpha \gamma_1} u, |D|^{-\alpha \gamma_2} v), \\
 B_{k, h}^{exp, 2}(u, v) &= |D|^{-\gamma_3} L_k(|D|^{-\gamma_1} u, |D|^{-\gamma_2} v), \\
  B_{k}^{lin, 2}(u, v) &= L_k(|D|^{ - 1} u, |D|^{1-\alpha} v).
\end{aligned}
\end{equation*}
\end{proposition}

\begin{proof}

From the first term of $\hat Q_k^2$ in \eqref{Qk-def}, we have the symbol
\[ m_k^{2,1}(\xi_1, \xi_2) := i\frac{\hat P_{k_1}^+(\xi_1 + \xi_2)\hat P_{\geq k}(\xi_1) \xi_2 \hat P_{< k}(\xi_2)}{\Omega(\xi_1, \xi_2) }. \]
Applying Lemma~\ref{l:approxOmega} with $|\xi_{min}| = |\xi_2|$ and $|\xi_{max}|\approx |\xi_1|$, we find
\[ |m_k^{2,1}(\xi_1, \xi_2)| \lesssim |\xi_1|^{-\alpha}. \]
In particular, since $|\xi_{max}| \approx |\xi_1|$,
\[ |m_k^{2,1}(\xi_1, \xi_2)| \lesssim |\xi_1 + \xi_2|^{-\alpha\gamma_3}|\xi_1|^{-\alpha \gamma_1} |\xi_2|^{-\alpha \gamma_2}. \]

\

From the second term of $\hat Q_k^2$ in \eqref{Qk-def}, we have the symbol
\[ m_k^{2,2}(\xi_1, \xi_2) := \half i (\xi_1 + \xi_2)\frac{\hat P_k^+(\xi_1 + \xi_2)\hat P_{\geq k}(\xi_1) \hat P_{\geq k}(\xi_2)}{\Omega(\xi_1, \xi_2) }. \]
Applying Lemma~\ref{l:approxOmega} with $|\xi_{min}| \approx |\xi_1 + \xi_2|$ and $|\xi_{max}|\approx |\xi_1| \approx |\xi_2|$, we find
\[ |m_k^{2,2}(\xi_1, \xi_2)| \lesssim |\xi_1|^{-\alpha} \]
as before.

\

From the third term of $\hat Q_k^2$ in \eqref{Qk-def}, we have the symbol
\[ m_k^{2,3}(\xi_1, \xi_2):= i\frac{[\hat P_k^+(\xi_1 + \xi_2) - \hat P_k^+(\xi_2)] \hat P_{(k', k)}(\xi_1) \xi_2}{\Omega(\xi_1, \xi_2)}. \]
Applying Lemma~\ref{l:approxOmega} with $|\xi_{min}| \approx |\xi_1|$ and $|\xi_{max}| \approx |\xi_2|$, we find
\[ |\Omega(\xi_1, \xi_2)| \approx | \xi_1| |\xi_2|^\alpha. \]
Since
\[ |\hat P_k^+(\xi_1 + \xi_2) - \hat P_k^+(\xi_2)| \lesssim \xi_1 2^{-k} \approx \xi_1 \xi_2^{-1}, \]
we conclude 
\[ |m_k^{2,3}(\xi_1, \xi_2)| \lesssim |\xi_2|^{-\alpha}. \]
This completes the analysis of $\hat Q_k^2$ and hence $B_k^2$. 

\

The analysis of $\hat Q_{k, I}^{exp, 2}$ is similar to the first term of  $\hat Q_k^2$, so we obtain the same form for $B_{k, I}^{exp, 2}$.

\

We similarly treat $\hat Q_{k, II}^{exp, 2}$ in \eqref{Qk-def}. We have the symbol
\[  m_{k, II}^{exp, 2}(\xi_1, \xi_2):= \frac{|\xi_1|^{\alpha}\hat{P}_{(k', k)}(\xi_1)\hat{P}_k^+(\xi_2)}{\Omega(\xi_1, \xi_2) }. \]
Applying Lemma~\ref{l:approxOmega} with $|\xi_{min}| \approx |\xi_1|$ and $|\xi_{max}| \approx |\xi_2|$, we find
\[ |m_{k, II}^{exp, 2}(\xi_1, \xi_2)| \lesssim |\xi_1|^{\alpha - 1}|\xi_2|^{-\alpha}. \]

\

Next, we consider $\hat Q_{k, h}^{exp, 2}$, which has the symbol
\[ m_{k, h}^{exp, 2}(\xi_1, \xi_2) := \frac{|(1 - h)\xi_1 + \xi_2|^{\alpha - 2} \xi_1 P_{(k', k)}(\xi_1) \xi_2 P_k^+(\xi_2)}{\Omega(\xi_1, \xi_2)}. \]
Applying Lemma~\ref{l:approxOmega} with $|\xi_{min}| \approx |\xi_1|$ and $|\xi_{max}| \approx |\xi_2|$, we find
\[ |m_{k, h}^{exp, 2}(\xi_1, \xi_2)| \lesssim |\xi_2|^{1-\alpha} |(1 - h)\xi_1 + \xi_2|^{\alpha - 2} \lesssim |\xi_2|^{-1}. \]

\

Lastly, for $\hat Q_{k}^{lin, 2}$, we have the symbol
\[  m_{k}^{lin, 2}(\xi_1, \xi_2):= \frac{- i\xi_2 \hat{P}_{(0, k)}(\xi_1)\hat{P}_k^+(\xi_2)}{\Omega(\xi_1, \xi_2) }. \]
Applying Lemma~\ref{l:approxOmega} with $|\xi_{min}| \approx |\xi_1|$ and $|\xi_{max}| \approx |\xi_2|$, we find
\[ |m_{k}^{lin, 2}(\xi_1, \xi_2)| \lesssim |\xi_1|^{- 1}|\xi_2|^{1-\alpha}. \]
\end{proof}


\section{Estimates for the full equation}\label{s:bootstrapStart}

In this section we prove a priori bounds for smooth solutions to the dispersion-generalized Benjamin-Ono equation \eqref{BO},
\[
(\D_t - |D|^\alpha \D_x) \phi = \half \D_x(\phi^2).
\]
Since \eqref{BO} admits the scaling symmetry
\begin{equation}\label{scaleInv}
\phi(t, x) \rightarrow \lambda^\alpha \phi(\lambda^{1 + \alpha} t, \lambda x),
\end{equation}
it suffices to work with solutions with small data and time interval $[0, 1]$. 

To state our main estimate, we define the Strichartz space
\[
S = L^\infty_t L^2_x \cap L^4_t W_x^{-\frac{1}{4}(1 - \alpha), \infty},
\]
as well as the lateral Strichartz norm,
\[
\|u\|_{S_{lat}} = \||D|^{-\frac14} u\|_{L^4_x L^\infty_t} + \||D|^{\frac{\alpha}{2}} u\|_{L^\infty_x L^2_t}.
\]

We will also state the estimate using the language of frequency envelopes, which will be a convenient formulation to prove local well-posedness and in particular the continuous dependence on initial data. Following Tao \cite{tao} (see also Ifrim-Tataru \cite{ITprimer}), we say that $\{c_k\}_{k = 0}^\infty \in \ell^2$ is a frequency envelope for $\phi \in H^s$ if
\begin{enumerate}[a)]
\item it satisfies the energy bound
\[\|P_k \phi\|_{H^s} \leq c_k,\]

\item it is slowly varying,
\[
\frac{c_j}{c_k} \lesssim 2^{\delta |j - k|},
\]
where $\delta$ is a small universal constant,

\item it satisfies the upper bound
\[
\sum_k c_k^2 \lesssim \| u\|^2_{H^s}.
\]
\end{enumerate}
Such frequency envelopes always exist, for instance by taking
\[
c_k = \sup_j 2^{-\delta |j - k|} c_j.
\]

\begin{theorem}\label{t:nonlinear}
Let $\alpha \in [1, 2]$, $s \geq \frac{3}{4}(1 - \alpha)$, and $\phi$ be a smooth solution to \eqref{BO} on $I = [0, 1]$ with small initial data,
\[
\|\phi_0\|_{H^s} \leq \eps.
\]
Let $\{c_k \}_{k = 0}^\infty \in \ell^2$ so that $\eps c_k$ is a frequency envelope for $\phi_0 \in H^s$, and denote $d_k = 2^{-sk} c_k$. Then we have

\begin{enumerate}[a)]
\item the Strichartz and lateral Strichartz bounds
\[
\|\phi_k\|_{S \cap S_{lat}} \lesssim \eps d_k,
\]
\item the bilinear bound
\[
\|\phi_j \phi_k\|_{L^2} \lesssim 2^{-\frac{\alpha}{2} \max(j,k)} \eps^2 d_j d_k, \qquad j \neq k.
\]
\end{enumerate}
\end{theorem}

\subsection{The bootstrap argument} 

Here, we set up the proof of Theorem~\ref{t:nonlinear} using a standard continuity argument. For $t_0 \in (0, 1]$, we define
\begin{equation*}
\begin{aligned}
M(t_0) &:= \sup_k d_k^{-2} \|\phi_k\|^2_{S([0, t_0]) \cap S_{lat}([0, t_0])} + \sup_{k \neq j \in \N} \sup_y 2^{\frac{\alpha}{2}\max(j, k)} d_{j}^{-1} d_{k}^{-1}\|\phi_{j} T_y\phi_{k}\|_{L^2([0, t_0])}.
\end{aligned}
\end{equation*}

Then to prove Theorem~\ref{t:nonlinear}, it suffices to show that $M(1) \lesssim \eps^2$. In turn, since $M$ is continuous in $t$ and
\[
\lim_{t \rightarrow 0} M(t) \lesssim \eps^2,
\]
we may use a continuity argument to reduce this to showing
\begin{equation}\label{nlinBootstrap}
M(t_0) \lesssim \eps^2 + (C\eps)^3
\end{equation}
under the bootstrap assumption
\begin{equation}\label{nlinBootstrapAssump}
M(t_0) \leq (C \eps)^2 \ll 1,
\end{equation}
where we choose $\eps$ sufficiently small depending on $C$. 

Recall that we define the control parameters \eqref{control-parameters},
\[
\AA(t) = \| \langle D \rangle^{\frac14 (1 - 3\alpha)} \phi(t)\|_{L^\infty}, \qquad \BB(t) = \| \langle D \rangle^{\half (1 - \alpha)} \phi(t)\|_{L^\infty}.
\]
We observe that in particular, the Strichartz component of the bootstrap assumption implies the pointwise estimate
\begin{equation}\label{inftyest}
\|\phi_k\|_{L^4_tL^\infty_x} \lesssim 2^{\frac14(1 - \alpha)k}C\eps d_k \leq 2^{-\half (1 - \alpha)k} C\eps c_k = 2^{-k'} C\eps c_k,
\end{equation}
and thus $\BB \in L^4_t$. On the other hand, by Sobolev embedding,
\begin{equation}\label{inftyest2}
\|\phi_k\|_{L^\infty_{t, x}} \lesssim 2^{\frac{k}{2}+}C\eps d_k = 2^{(\frac12 - s)k+} C\eps c_k \leq 2^{\frac14(3\alpha - 1)k} C\eps c_k,
\end{equation}
and thus $\AA \in L^\infty_t$.

On the other hand, from the lateral estimates, we have
\begin{equation}\label{latest}
\begin{aligned}
&\|\phi_k\|_{L^4_xL^\infty_t} \lesssim 2^{\frac{k}{4}} C \eps d_k \leq 2^{(\frac34\alpha - \half)k} C \eps c_k, \\
&\|\phi_k\|_{L^\infty_xL^2_t} \lesssim 2^{-\frac{\alpha}{2}k} C \eps d_k \leq 2^{\frac14(\alpha - 3)k} C \eps c_k.
\end{aligned}
\end{equation}

To prove Theorem~\ref{t:nonlinear}, it remains to prove \eqref{nlinBootstrap} under the bootstrap assumption \eqref{nlinBootstrapAssump}.

\subsection{The combined renormalization}

We reduce the estimate \eqref{nlinBootstrap} on the original solution $\phi$ to that for an unknown combining the two renormalizations discussed in Section~\ref{s:nf}.

Beginning with equation \eqref{expError} for the conjugated variable $\psi_k^+$, and substituting the exponential commutator expansion of Lemma~\ref{l:expCom}, we have
\begin{equation}\label{expError2}
\begin{aligned}
((\D_t - |D|^\alpha \D_x) - \phi_{\leq k'} \D_x) \psi_k^+ &= P^+_k (e^{iA} Q_k^2(\phi, \phi) - [ \phi_{\leq k'} \D_x, e^{iA}] \phi_k^+ \\
&\quad - (1 + \alpha)^{-1}e^{iA} Q_{k, I}^{exp, 2}(\phi, \phi) \\
&\quad + (1 + \alpha)^{-1} |D|^{-\alpha} e^{iA}\D_x Q_{k, II}^{exp, 2}(\phi, \phi) \\
&\quad -\frac{\alpha(\alpha - 1)}{2(1 + \alpha)} \D_x |D|^{-\alpha}e^{iA}\int_0^1(1 - h) Q_{k, h}^{exp, 2}(\phi, \phi) \, dh \\
&\quad + Q_k^{exp, 3+}(\phi_k^+)).
\end{aligned}
\end{equation}

Then we define the combined renormalization
\[
\tilde \psi_k^+ := \psi_k^+ - B_k(\phi, \phi)
\]
where
\begin{equation}\label{finalNF}
\begin{aligned}
B_k(\phi, \phi) &:= P^+_k \bigg(e^{iA} B_k^2(\phi, \phi) - (1 + \alpha)^{-1}e^{iA} B_{k, I}^{exp, 2}(\phi, \phi) \\
&\quad \phantom{\psi_k^+ - P^+_k \Big[} + (1 + \alpha)^{-1} |D|^{-\alpha} e^{iA}\D_x B_{k, II}^{exp, 2}(\phi, \phi) \\
&\quad \phantom{\psi_k^+ - P^+_k \Big[} -\frac{\alpha(\alpha - 1)}{2(1 + \alpha)} \D_x |D|^{-\alpha}e^{iA}\int_0^1(1 - h) B_{k, h}^{exp, 2}(\phi, \phi) \, dh\bigg),
\end{aligned}
\end{equation}
so that the quadratic contributions of the normal form corrections precisely cancel the quadratic terms on the right hand side of \eqref{expError2}. This leaves only perturbative cubic and higher order contributions, which arise from 
\begin{itemize}
\item commutators with the exponential operator $e^{iA}$, 
\item terms containing $\D_x (\phi^2)$, introduced by using the equation \eqref{BO} to replace time derivatives, and 
\item one instance of $\phi_{\leq k'} \D_x$ applied to the correction $B_k$.
\end{itemize}
Precisely, the combined renormalization $\tilde \psi_k^+$ satisfies
\begin{equation}\label{finalError}
((\D_t - |D|^\alpha \D_x) - \phi_{\leq k'} \D_x) \tilde \psi_k^+ = Q_k
\end{equation}
where 
\begin{equation*}
\begin{aligned}
Q_k &:= P^+_k \bigg(Q_k^{exp, 3+}(\phi_k^+) - [\phi_{\leq k'} \D_x, e^{iA}] \phi_k^+ \\
&\quad \phantom{P^+_k \Big[} -[(\D_t - |D|^\alpha \D_x), e^{iA}] B_k^2(\phi, \phi) \\
&\quad \phantom{P^+_k \Big[} +(1 + \alpha)^{-1}[(\D_t - |D|^\alpha \D_x), e^{iA}] B_{k, I}^{exp, 2}(\phi, \phi) \\
&\quad \phantom{P^+_k \Big[} -(1 + \alpha)^{-1}|D|^{-\alpha} [(\D_t - |D|^\alpha \D_x), e^{iA}]\D_x B_{k, II}^{exp, 2}(\phi, \phi) \\
&\quad \phantom{P^+_k \Big[} +\frac{\alpha(\alpha - 1)}{2(1 + \alpha)}\D_x |D|^{-\alpha} [(\D_t - |D|^\alpha \D_x), e^{iA}]\int_0^1(1 - h) B_{k, h}^{exp, 2}(\phi, \phi) \, dh \bigg) \\
&\quad- \half (B_k(\D_x (\phi^2), \phi) + B_k(\phi, \D_x (\phi^2))) \\
&\quad + \phi_{\leq k'} \D_x B_k(\phi, \phi).
\end{aligned}
\end{equation*}

\subsection{Bounds on the normal form variable}

In this section we reduce the bootstrap for $\phi$ to the same problem for the normal form variable $\tilde \psi_k^+$. Precisely, we reduce \eqref{nlinBootstrap} to the same estimate for the renormalized variable $\tilde \psi_k^+$, and on the other hand, also show that the bootstrap assumption \eqref{nlinBootstrapAssump} for $\phi$ implies estimates on the initial data for $\tilde \psi_k^+$.

\

We first establish estimates on the bilinear correction $B_k$ defined in \eqref{finalNF}.

\begin{lemma}\label{l:Bk-bootstrap}
a) We have
\[
\|B_k(\phi, \phi)\|_{S \cap S_{lat}}\lesssim 2^{-\frac14(1 + \alpha)k} (C\eps)^2 d_k.
\]

b) For $j \geq k$, we may write
\[
B_k(\phi_j, \phi_j) = 2^{-\alpha j} L_k(\phi_j, \phi_j)
\] 
and estimate
\[
\|B_k(\phi_{\geq j}, \phi_{\geq j})\|_{S \cap S_{lat}}\lesssim 2^{-\frac14(1 + \alpha)j} (C\eps)^2 d_k.
\]

c) We have the initial data estimate
\[
\|B_k(\phi(0), \phi(0))\|_{H^s} \lesssim  2^{-\frac14(\alpha + 1)k} (C\eps)^2 c_k.
\] 
\end{lemma}

\begin{proof}

$a)$ The estimate combines bounds on the exponential conjugation and bilinear forms in, respectively, Propositions~\ref{p:expBd} and \ref{p:Bk-est}, along with the bootstrap assumption \eqref{nlinBootstrapAssump}. For instance, consider the third term of $B_k$, defined in \eqref{finalNF},
\[
|D|^{-\alpha} e^{iA}\D_x B_{k, II}^{exp, 2}(\phi, \phi).
\]
We use Proposition~\ref{p:Bk-est} to rewrite this with the $L$ notation. Here, the inputs to $B_{k, II}^{exp, 2}$ are already at frequency $2^k$ or lower, so we choose to put the derivative gain on the output at frequency $2^k$ (by setting $\gamma_3 = 1$) and write this as 
\[
|D|^{-\alpha} e^{iA}\D_x |D|^{-\alpha} L_k(|D|^{\alpha - 1} \phi, \phi).
\]
We then use Proposition~\ref{p:expBd} and the bootstrap assumption \eqref{nlinBootstrapAssump} to estimate
\begin{equation*}
\begin{aligned}
\||D|^{-\alpha} e^{iA}\D_x |D|^{-\alpha} L_k(|D|^{\alpha - 1} \phi, \phi)\|_S &\lesssim 2^{(1 - 2\alpha)k} \| |D|^{\alpha - 1} \phi_{(k', k)}\|_{L^\infty} \|\phi_k\|_S \\
&\lesssim 2^{(1 - 2\alpha)k} \cdot 2^{(\alpha - 1)k}  2^{\frac14(3\alpha - 1)k} C\eps c_k \cdot \|\phi_k\|_S \\
&\lesssim 2^{-\frac14(\alpha + 1)k} (C\eps)^2 d_k
\end{aligned}
\end{equation*}
as desired. The other terms of $B_k$ arising due to the exponential conjugation (the second and fourth terms of $B_k$) are simpler and estimated similarly. The discussion applies as well to $S_{lat}$.

The analysis also applies similarly to the first term $B_k^2$ of $B_k$ in \eqref{finalNF}, except here the inputs may include frequencies higher than $2^k$, so in such cases it is more economical to use Proposition~\ref{p:Bk-est} in a way that puts the derivative gain on the inputs. For instance, consider the second term in $B_k^2$ (provided in \eqref{Qk-def} in terms of $Q_k^2$), which consists of balanced frequencies higher than $2^k$. We set $\gamma_2 = 1$ in Proposition~\ref{p:Bk-est}, so that for inputs at frequency $2^j$, this term has the form
\[
2^{-\alpha j} e^{iA} L_k(\phi_j, \phi_j).
\]
Then we may estimate (using \eqref{inftyest2} for the second line and the slowly varying property of the frequency envelope in the last line)
\begin{equation*}
\begin{aligned}
\| e^{iA}\sum_{j \geq k}2^{-\alpha j} L_k(\phi_j, \phi_j)\|_S &\lesssim \sum_{j \geq k} 2^{-\alpha j}\| \phi_j\|_{L^\infty}  \|\phi_j\|_S \\
&\lesssim \sum_{j \geq k} 2^{-\alpha j}\cdot 2^{(\frac12 - s)j+} C\eps c_j  \cdot \|\phi_j\|_S \\
&\lesssim 2^{-\alpha k + (\half - \frac{3}{4}(1 - \alpha))k} (C\eps)^2 d_k \\
&= 2^{-\frac14(\alpha + 1)k} (C\eps)^2 d_k
\end{aligned}
\end{equation*}
as desired.

\

$b)$ We set $\gamma_2 = 1$ in Proposition~\ref{p:Bk-est} to obtain the derivative gain $2^{-\alpha j}$, and use Proposition~\ref{p:expBd} to absorb instances of the operator $e^{iA}$.

The proof of the estimate is similar to the discussion from part $a)$.

\

$c)$ The proof is similar to the proof of $a)$ using Propositions~\ref{p:expBd} and \ref{p:Bk-est}. For instance, considering again the third term of $B_k$,
\begin{equation*}
\begin{aligned}
\||D|^{-\alpha} e^{iA}\D_x |D|^{-\alpha} L_k(|D|^{\alpha - 1} \phi, \phi)\|_{H^s} &\lesssim 2^{(1 - 2\alpha)k} \| |D|^{\alpha - 1} \phi_{(k', k)}\|_{L^\infty} \|\phi_k\|_{H^s} \\
&\lesssim 2^{(1 - 2\alpha)k} \cdot 2^{(\alpha - 1)k}  2^{\frac14(3\alpha - 1)k} C\eps c_k \cdot \|\phi_k\|_{H^s} \\
&\lesssim 2^{-\frac14(\alpha + 1)k} (C\eps)^2 c_k.
\end{aligned}
\end{equation*}

\end{proof}

We now reduce the proof of the bootstrap estimate \eqref{nlinBootstrap} to the proof of the same estimate on $\tilde \psi_k^+$:
\begin{lemma} 

a) Assume the Strichartz bounds
\begin{equation}\label{psi-bootstrap-strich}
\begin{aligned}
d_k^{-2}\|\tilde \psi_k^+\|_{S \cap S_{lat}}^2 \lesssim \eps^2 + (C\eps)^3.
 \end{aligned}
\end{equation}
Then the same estimate holds for $\phi_k^+$. 

\

b) Assume the bilinear bounds
\begin{equation}\label{psi-bootstrap}
\begin{aligned}
d_{j}^{-1} d_{k}^{-1} 2^{\frac{\alpha}{2}\max(j, k)} \|\tilde \psi_{j}^+ \tilde \psi_{k}^+\|_{L^2} \lesssim \eps^2 + (C\eps)^3.
 \end{aligned}
\end{equation}
Then the same estimate holds for $\phi_{j}^+ \phi_{k}^+$.

\

c) Given the bootstrap assumption \eqref{nlinBootstrapAssump}, we have
\begin{equation}\label{psi-data}
\|\tilde \psi_k^+(0)\|_{H^s} \lesssim \eps c_k.
\end{equation}
\end{lemma}

\begin{proof}
$a)$ From the definition above \eqref{finalNF}, we have
\[
\tilde \psi_k^+ = \psi_k^+ - B_k(\phi, \phi).
\]
Lemma~\ref{l:Bk-bootstrap} provides the sufficient bound for $B_k$, using that 
\[
d_k^{-1} \|B_k(\phi, \phi)\|_{S \cap S_{lat}} \lesssim 2^{-\frac14(1 + \alpha)k} (C\eps)^2 \ll (C\eps)^2
\]
for large $2^k$.

\

$b)$ We have
\[
\tilde \psi_{j}^+ \tilde \psi_{k}^+ = (\psi_j^+ - B_j(\phi, \phi)) (\psi_k^+ - B_k(\phi, \phi)).
\]
We estimate the cubic and higher terms to first reduce to $\psi_j^+ \psi_k^+$. Consider in particular 
\[
\psi_j^+B_k(\phi, \phi)
\]
with the other cases being similar. 

We first consider the case $j > k$, which we in turn reduce to two cases. When the inputs are also at frequency $\geq j$, we use the bootstrap assumption \eqref{inftyest} with Proposition~\ref{p:expBd} on $\psi_j^+$, and $b)$ of Lemma~\ref{l:Bk-bootstrap} on $B_k$:
\begin{equation*}
\begin{aligned}
\|\psi_j^+B_k(\phi_{\geq j}, \phi_{\geq j})\|_{L^2} \lesssim \|\psi_j^+\|_{L^4L^\infty} \|B_k(\phi_{\geq j}, \phi_{\geq j})\|_{S} &\lesssim 2^{\frac14(1 - \alpha)j}C\eps d_j \cdot 2^{-\frac14(1 + \alpha)j}(C\eps)^2 d_k\\
&= 2^{-\frac{\alpha}{2}j} (C\eps)^3 d_jd_k
\end{aligned}
\end{equation*}
which suffices.

It remains to estimate 
\begin{equation}\label{bjest}
\psi_j^+B_k(\phi_{< j}, \phi_{< j}),
\end{equation}
for which we use the bilinear bootstrap estimate. Here, we use Proposition~\ref{p:expBd} to apply bilinear estimates though the $e^{iA}$ conjugations, along with the gain from $b)$ of Lemma~\ref{l:Bk-bootstrap} to compensate for the loss from the pointwise estimate on the remaining third variable:
\begin{equation*}
\begin{aligned}
\|\psi_j^+B_k(\phi_{\ell}, \phi_{\ell})\|_{L^2} &= 2^{-\alpha \ell}\|\psi_j^+e^{iA} L_k(\phi_{\ell}, \phi_{\ell})\|_{L^2} \\
&\lesssim 2^{-\alpha \ell} \|L(\psi_j, \phi_{\ell}, \phi_{\ell})\|_{L^2} \\
&\lesssim 2^{-\alpha \ell}\cdot2^{-\frac{\alpha}{2}j}(C\eps)^2 d_jd_{\ell} \cdot 2^{\frac14(3\alpha - 1)\ell}C\eps c_{\ell} \\
&= 2^{-\frac{\alpha}{2}j}(C\eps)^3 d_jd_{\ell} \cdot 2^{-\frac14(1 + \alpha)\ell}.
\end{aligned}
\end{equation*}
Then summing, we obtain
\[
\|\psi_j^+B_k(\phi_{k < \cdot < j}, \phi_{k < \cdot < j})\|_{L^2} \lesssim 2^{-\frac{\alpha}{2}j}(C\eps)^3 d_jd_{k}
\]
which suffices. A similar bilinear estimate for 
\[
\psi_j^+B_k(\phi_{k}, \phi_{< k}), \qquad \psi_j^+B_k(\phi_{< k}, \phi_{k})
\]
addresses the remaining components of \eqref{bjest}. A similar case analysis using the bilinear bootstrap estimate can be applied to address case when $j \leq k$. 

To conclude, we use Proposition~\ref{p:expBd} to reduce from $\psi_j^+\psi_k^+$ to $\phi_j^+\phi_k^+$.


\

$c)$ We use part $c)$ of Lemma~\ref{l:Bk-bootstrap} to estimate $B_k$ and Proposition~\ref{p:expBd} to absorb $e^{iA}$.

\end{proof}

\subsection{Bounds on the source terms}

In this section, we bound the cubic and higher source terms $Q_k$ of the equation \eqref{finalError} for $\psi_k^+$.

\begin{lemma}\label{l:source-bd}
The source terms $Q_k$ of \eqref{finalError} satisfy the estimate
\[
\|Q_k\|_{L^1_tL^2_x} \lesssim (C\eps)^3 d_k.
\]
\end{lemma}

\begin{proof}

We consider the terms in $Q_k$, starting with $Q_k^{exp, 3+}(\phi_k^+)$, which is provided in Lemma~\ref{l:expCom}. Here we demonstrate an estimate for the third term of $Q_k^{exp, 3+}(\phi_k^+)$, which is typical and possesses the full range of possible frequencies for the input $\phi$'s:
\[
|D|^{-\alpha}  e^{iA} \D_x P_{(k', k)} (\phi^2) \phi_k^+ = |D|^{-\alpha}  e^{iA} \D_x L_k(\phi, \phi, \phi_k^+).
\]
Using Proposition~\ref{p:expBd}, it suffices to show
\begin{equation}\label{source-est}
\|L_k(\phi, \phi, \phi_k^+) \|_{L^1_tL^2_x} \lesssim  2^{k(\alpha - 1)}(C\eps)^3 d_k.
\end{equation}

\

For the low frequency component on the first two inputs, we use Strichartz twice:
\begin{equation*}
\begin{aligned}
\|L_k(\phi_{\leq k}, \phi_{\leq k}, \phi_k^+)\|_{L^1_tL^2_x} &\lesssim \|\phi_{\leq k}\|_{L^4_t L^\infty_x}^2 \|\phi_k^+\|_{L^\infty_t L^2_x} \lesssim 2^{-(1 - \alpha)k} (C\eps)^2 \cdot C\eps d_k
\end{aligned}
\end{equation*}
as desired.

For the remaining high frequency components, let $j > k$ and consider
\[L_k(\phi_{j}, \phi_{j}, \phi_k^+).\]
When $\alpha \leq \frac53$, we use bilinear estimates on the latter two variables and Strichartz for the first:
\begin{equation*}
\begin{aligned}
\|L_k(\phi_{j}, \phi_{j}, \phi_k^+)\|_{L^1_tL^2_x} &\lesssim  2^{-\half(1 - \alpha)j} C\eps c_{j} \cdot 2^{-\frac{\alpha}{2} j}(C\eps)^2 d_{j} d_k \\
&\lesssim 2^{-\half j} 2^{-\frac34(1 - \alpha)j} (C\eps)^3 d_k \\
&= 2^{(-\frac54 + \frac34\alpha) j} (C\eps)^3 d_k \\
&\lesssim (C\eps)^3 d_k.
\end{aligned}
\end{equation*}
When $\alpha \geq \frac54$, and in particular when $\alpha > \frac53$, we can use lateral Strichartz estimates, estimating $\phi_k^+$ in $L^4_x L^\infty_t$, and both instances of $\phi_j$ in $L^\infty_x L^2_t$. Precisely, using the bootstrap estimate \eqref{latest},
\begin{equation*}
\begin{aligned}
\|L_k(\phi_{j}, \phi_{j}, \phi_k^+)\|_{L^1_tL^2_x} &\lesssim \|L_k(\phi_{j}, \phi_{j}, \phi_k^+)\|_{L^2_{t, x}} \\
&\lesssim  2^{\frac12(\alpha - 3)j} (C \eps c_j)^2 \cdot 2^{\frac{k}{4}} C \eps d_k \\
&\lesssim 2^{\frac12(\alpha - 3)j} 2^{\frac{k}{4}} (C\eps)^3 d_k \\
&\lesssim 2^{k(\alpha - 1)}(C\eps)^3 d_k.
\end{aligned}
\end{equation*}

\

The other terms of $Q_k$ are treated similarly using Proposition~\ref{p:expBd} to treat the exponential conjugations, and Proposition~\ref{p:Bk-est} to estimate the bilinear forms. We also use Lemma~\ref{l:expCom} to see that $[(\D_t - |D|^\alpha \D_x), e^{iA}]$ is an operator of order 1.

\end{proof}

\subsection{Closing the bootstrap}

To complete the proof of the bootstrap estimate \eqref{nlinBootstrap}, it remains to prove the estimates \eqref{psi-bootstrap-strich} and \eqref{psi-bootstrap} on $\tilde \psi_k^+$. For the former, we apply Theorem~\ref{t:strichartz-transport} and a straightforward energy estimate to obtain
\[
\|\tilde \psi_k^+\|_{S \cap S_{lat}} \lesssim \|\tilde \psi_k^+(0)\|_{L^2} + \|Q_k\|_{L^1L^2}.
\]
Using \eqref{psi-data} for $\tilde \psi_k^+(0)$ and Lemma~\ref{l:source-bd} for $Q_k$, we obtain
\begin{equation}\label{psi-bootstrap-strich-pf}
\|\tilde \psi_k^+\|_{S \cap S_{lat}} \lesssim \eps d_k + (C\eps)^3 d_k
\end{equation}
as desired.

For the bilinear estimate \eqref{psi-bootstrap}, we likewise apply Proposition~\ref{p:bilinear} to obtain
\begin{equation*}
\begin{aligned}
 \|\tilde \psi_j^+\tilde \psi_k^+\|_{L^2} &\lesssim 2^{-\frac{\alpha}{2}\max(j, k)}(\|\tilde \psi_j^+\|_{L^\infty L^2}\|\tilde \psi_k^+\|_{L^\infty L^2)} \\
&\quad + \|Q_j\|_{L^1L^2}\|\tilde \psi_k^+\|_{L^\infty L^2}  + \|Q_k\|_{L^1L^2}\|\tilde \psi_j^+\|_{L^\infty L^2}). 
\end{aligned}
\end{equation*}
Then using \eqref{psi-bootstrap-strich-pf} and Lemma~\ref{l:source-bd}, we obtain
\[
 \|\tilde \psi_j^+\tilde \psi_k^+\|_{L^2} \lesssim 2^{-\frac{\alpha}{2}\max(j, k)}(\eps^2 + (C\eps)^3 \eps  + (C\eps)^6 )d_j d_k
\]
as desired.

\section{Estimates for the linearized equation}

In this section we prove a priori bounds for the linearized equation for \eqref{BO},
\begin{equation}\label{linBO}
 (\D_{t} - |D|^\alpha \D_x) v = \D_x (\phi v). 
\end{equation}
Before stating the theorem, we remark that a frequency envelope $\{c_k \}_{k = 0}^\infty$ can serve as the frequency envelope for two functions $\phi_0$ and $v_0$ simultaneously by taking the maximum of two individual envelopes.

\begin{theorem}\label{t:linear}
Let $\alpha \in [1, 2]$, $s \geq \frac{3}{4}(1 - \alpha)$, and $\phi$ be a smooth solution to \eqref{BO} on $I = [0, 1]$ with small initial data,
\[
\|\phi_0\|_{H^s} \leq \eps.
\]
Further, let $v$ be an $H^{s - \frac12}$ solution to \eqref{linBO}, and $\{c_k \}_{k = 0}^\infty \in \ell^2$ so that $\eps c_k$ is a frequency envelope for $\phi_0 \in H^s$ and $c_k$ is a frequency envelope for $v_0 \in H^{s - \frac12}$. Also denote $d_k = 2^{-sk} c_k$. Then we have

\begin{enumerate}[a)]
\item the Strichartz and lateral Strichartz bounds
\[
\|v_k\|_{S \cap S_{lat}} \lesssim 2^\frac{k}{2} d_k,
\]
\item the bilinear bound
\[
\|v_j \phi_k\|_{L^2} \lesssim 2^\frac{j}{2} 2^{-\frac{\alpha}{2} \max(j,k)} \eps d_j d_k, \qquad j \neq k.
\]
\end{enumerate}
\end{theorem}

The proof of Theorem~\ref{t:linear} largely follows that of the corresponding nonlinear result, Theorem~\ref{t:nonlinear}, using a combined normal form analysis with a reduction to a bootstrap argument. We review the steps below.

\subsection{The bootstrap argument} 

We proceed in the same manner as for the nonlinear equation, setting up the proof of Theorem~\ref{t:linear} using a standard continuity argument. For $t_0 \in (0, 1]$, we define
\begin{equation*}
\begin{aligned}
M(t_0) &:= \sup_k 2^{-\frac{k}{2}}d_k^{-1} \|v_k\|_{S([0, t_0]) \cap S_{lat}([0, t_0])} + \sup_{k \neq j \in \N} \sup_y \eps^{-1} 2^{-\frac{j}{2}}2^{\frac{\alpha}{2}\max(j, k)} d_{j}^{-1} d_{k}^{-1}\|v_{j} T_y\phi_{k}\|_{L^2([0, t_0])}.
\end{aligned}
\end{equation*}

Then Theorem~\ref{t:linear} reduces to showing
\begin{equation}\label{linBootstrap}
M(t_0) \lesssim 1 + \eps C
\end{equation}
under the bootstrap assumption
\begin{equation}\label{linBootstrapAssump}
M(t_0) \leq C.
\end{equation}

\subsection{The combined renormalization}

In analogy with the nonlinear counterpart \eqref{BO2}, we rewrite \eqref{linBO} as a frequency-localized paradifferential equation, 
\begin{equation*}
(\D_t - |D|^\alpha \D_x) v_k^+ - \phi_{< k} \D_x v_k^+ = Q_k^2(\phi, v) + Q_k^2(v, \phi) + v_{< k} \D_x \phi_k^+.
\end{equation*}

Our analysis proceeds in a way similar to the nonlinear equation. We apply an exponential conjugation to reduce the order of most of the paradifferential quadratic term $\phi_{< k} \D_x v_k^+$ on the left hand side. The residual terms will include perturbative cubic terms, a quadratic term which may be viewed as a transport term with very low frequency coefficient, and better-balanced quadratic terms. 

These residual quadratic terms, along with most of the existing ones enumerated above, can be treated via paradifferential normal forms. However, observe that here we have an extra quadratic term $ v_{< k} \D_x \phi_k^+$ relative to the nonlinear analysis. This will also be treated in part with a normal form, but the component
\[
v_{< 0} \D_x \phi_k^+
\] 
can be treated perturbatively using bilinear estimates.

\

We begin with the exponential conjugation. Recall that $\Phi$ and $a$ are defined in \eqref{Phi-def} and \eqref{a-def}, respectively. Then we define the exponentially conjugated linearized variable,
\begin{equation}\label{w-def}
w_k^+ := P^+_k e^{iA} v_k^+
\end{equation}
which satisfies
\begin{equation}\label{expErrorLin}
\begin{aligned}
(\D_t - |D|^\alpha \D_x - \phi_{\leq k'} \D_x) w_k^+ &= P^+_k( e^{iA} (Q_k^2(\phi, v) + Q_k^2(v, \phi) + \phi_{(k', k)} \D_x v_k^+ + v_{< k} \D_x \phi_k^+) \\
&\quad + [(\D_t - |D|^\alpha \D_x) - \phi_{\leq k'} \D_x, e^{iA}] v_k^+).
\end{aligned}
\end{equation}
Using Lemma~\ref{l:expCom}, the commutator with $(\D_t - |D|^\alpha \D_x)$ exhibits cancellation with the first order quadratic term $\phi_{(k',k)} \D_x v_k^+$. Precisely, we have
\begin{equation}\label{expError2lin}
\begin{aligned}
(\D_t - |D|^\alpha \D_x - \phi_{\leq k'} \D_x) w_k^+ &= P^+_k( e^{iA} (Q_k^2(\phi, v) + Q_k^2(v, \phi) + v_{< k} \D_x \phi_k^+) \\
&\quad - [\phi_{\leq k'} \D_x, e^{iA}] v_k^+ \\
&\quad - (1 + \alpha)^{-1}e^{iA} Q_{k, I}^{exp, 2}(\phi, v) \\
&\quad + (1 + \alpha)^{-1} |D|^{-\alpha} e^{iA}\D_x Q_{k, II}^{exp, 2}(\phi, v) \\
&\quad -\frac{\alpha(\alpha - 1)}{2(1 + \alpha)} \D_x |D|^{-\alpha}e^{iA}\int_0^1(1 - h) Q_{k, h}^{exp, 2}(\phi, v) \, dh \\
&\quad + Q_k^{exp, 3+}(v_k^+)).
\end{aligned}
\end{equation}

Then we define the combined renormalization
\[
\tilde w_k^+ := w_k^+ - B_k^{lin}(\phi, v)
\]
where
\begin{equation}\label{finalNFlin}
\begin{aligned}
B_k^{lin}(\phi, v) &:= P^+_k \bigg(e^{iA} (B_k^2(\phi, v) + B_k^2(v, \phi) + B_k^{lin, 2}(v, \phi)) \\
&\quad \phantom{\psi_k^+ - P^+_k \Big[} - (1 + \alpha)^{-1}e^{iA} B_{k, I}^{exp, 2}(\phi, v) \\
&\quad \phantom{\psi_k^+ - P^+_k \Big[} + (1 + \alpha)^{-1} |D|^{-\alpha} e^{iA}\D_x B_{k, II}^{exp, 2}(\phi, v) \\
&\quad \phantom{\psi_k^+ - P^+_k \Big[} -\frac{\alpha(\alpha - 1)}{2(1 + \alpha)} \D_x |D|^{-\alpha}e^{iA}\int_0^1(1 - h) B_{k, h}^{exp, 2}(\phi, v) \, dh\bigg),
\end{aligned}
\end{equation}
so that the quadratic contributions of the normal form corrections cancel most of the quadratic terms on the right hand side of \eqref{expError2lin}. Similar to the setting of the corrected nonlinear variable $\tilde \psi_k^+$, this leaves perturbative cubic and higher order contributions, which arise from 
\begin{itemize}
\item commutators with the exponential operator $e^{iA}$, 
\item terms containing $\D_x(\phi^2)$ and $\D_x (\phi v)$, introduced by using the equation \eqref{BO} and the linearized equation \eqref{linBO} respectively to replace time derivatives, and 
\item one instance of $\phi_{\leq k'} \D_x$ applied to the correction $B_k^{lin}$. 
\end{itemize}
The main difference with the corrected nonlinear variable $\tilde \psi_k^+$ is that here, we have a quadratic remainder,
\begin{itemize}
\item the contribution from the component $v_{< 0} \D_x \phi_k^+$ of the term $v_{< k} \D_x \phi_k^+$ on the right hand side of \eqref{expError2lin}, which is left uncorrected.
\end{itemize}
Precisely, the combined renormalization $\tilde w_k^+$ satisfies
\begin{equation}\label{wfinalError}
(\D_t - |D|^\alpha \D_x - \phi_{\leq k'} \D_x) \tilde w_k^+ = Q_k^{lin}
\end{equation}
where 
\begin{equation*}
\begin{aligned}
Q_k^{lin} &:= P^+_k \bigg(e^{iA}( v_{< 0} \D_x \phi_k^+) + Q_k^{exp, 3+}(v_k^+) - [\phi_{\leq k'} \D_x, e^{iA}] v_k^+ \\
&\quad \phantom{P^+_k \Big[} -[(\D_t - |D|^\alpha \D_x), e^{iA}] (B_k^2(\phi, v) + B_k^2(v, \phi) + B_k^{lin, 2}(v, \phi)) \\
&\quad \phantom{P^+_k \Big[} +(1 + \alpha)^{-1}[(\D_t - |D|^\alpha \D_x), e^{iA}] B_{k, I}^{exp, 2}(\phi, v) \\
&\quad \phantom{P^+_k \Big[} -(1 + \alpha)^{-1}|D|^{-\alpha} [(\D_t - |D|^\alpha \D_x), e^{iA}]\D_x B_{k, II}^{exp, 2}(\phi, v) \\
&\quad \phantom{P^+_k \Big[} +\frac{\alpha(\alpha - 1)}{2(1 + \alpha)}\D_x |D|^{-\alpha} [(\D_t - |D|^\alpha \D_x), e^{iA}]\int_0^1(1 - h) B_{k, h}^{exp, 2}(\phi, v) \, dh \bigg) \\
&\quad- \half (B_k^{lin}(\D_x (\phi^2), v) + B_k^{lin}(\phi, \D_x (\phi v))) \\
&\quad + \phi_{\leq k'} \D_x B_k^{lin}(\phi, v).
\end{aligned}
\end{equation*}

\subsection{Bounds on the normal form variable}

In this section we reduce the bootstrap for $v$ to the same problem for the normal form variable $\tilde w_k^+$. Precisely, we reduce \eqref{linBootstrap} to the same estimate for the renormalized variable $\tilde w_k^+$, and on the other hand, also show that the bootstrap assumption \eqref{linBootstrapAssump} for $v$ implies estimates on the initial data for $\tilde w_k^+$.

\

We first establish estimates on the bilinear correction $B_k^{lin}$ defined in \eqref{finalNFlin}.

\begin{lemma}\label{l:Bk-bootstraplin}
a) We have
\[
\|B_k^{lin}(\phi, v)\|_{S \cap S_{lat}}\lesssim 2^{\frac14(1 - \alpha)k} C \eps d_k.
\]

b) For $j \geq k$, we may write
\[
B_k^{lin}(\phi_j, v_j) = 2^{-\alpha j} L_k(\phi_j, v_j)
\] 
and estimate
\[
\|B_k^{lin}(\phi_{\geq j}, v_{\geq j})\|_{S \cap S_{lat}}\lesssim 2^{\frac14(1 - \alpha)j} C \eps d_k.
\]

c) We have the initial data estimate
\[
\|B_k^{lin}(\phi(0), v(0))\|_{H^s} \lesssim 2^{\frac14(1 - \alpha)k} C\eps c_k.
\] 
\end{lemma}

\begin{proof}

The proof is similar to that of Lemma~\ref{l:Bk-bootstrap}, except with the linearized variable $v$ appearing in place of the second $\phi$ in the appropriate instances. We consider the proof of $a)$, which is typical.

Like the nonlinear setting, the estimate combines bounds on the exponential conjugation and bilinear forms in, respectively, Propositions~\ref{p:expBd} and \ref{p:Bk-est}, along with Theorem~\ref{t:nonlinear} to estimate $\phi$ and the bootstrap assumption \eqref{linBootstrapAssump} to estimate $v$. For instance, consider the fifth term of $B_k^{lin}$, defined in \eqref{finalNFlin},
\[
|D|^{-\alpha} e^{iA}\D_x B_{k, II}^{exp, 2}(\phi, v).
\]
We use Proposition~\ref{p:Bk-est} to rewrite with the $L$ notation, putting the derivative gain on the output at frequency $2^k$ by setting $\gamma_3 = 1$,
\[
|D|^{-\alpha} e^{iA}\D_x |D|^{-\alpha} L_k(|D|^{\alpha - 1} \phi, v).
\]
We then use Proposition~\ref{p:expBd}, Theorem~\ref{t:nonlinear}, and the bootstrap assumption \eqref{linBootstrapAssump} to estimate
\begin{equation*}
\begin{aligned}
\||D|^{-\alpha} e^{iA}\D_x |D|^{-\alpha} L_k(|D|^{\alpha - 1} \phi, v)\|_S &\lesssim 2^{(1 - 2\alpha)k} \| |D|^{\alpha - 1} \phi_{(k', k)}\|_{L^\infty} \|v_k\|_S \\
&\lesssim 2^{(1 - 2\alpha)k} \cdot 2^{(\alpha - 1)k}  2^{\frac14(3\alpha - 1)k} \eps c_k \cdot C2^\frac{k}{2} d_k \\
&\lesssim 2^{\frac14(1 - \alpha)k} C \eps d_k
\end{aligned}
\end{equation*}
as desired. The rest of the proof is a similar adaptation of the proof of Lemma~\ref{l:Bk-bootstrap}, with $\|v_k\|_S$ in the place of $\|\phi_k\|_S$.

\end{proof}

We now reduce the proof of the bootstrap estimate \eqref{linBootstrap} to the proof of the same estimate on $\tilde w_k^+$:
\begin{lemma} 

a) Assume the Strichartz bounds
\begin{equation}\label{w-bootstrap-strich}
\begin{aligned}
2^{-k}d_k^{-2}\|\tilde w_k^+\|_{S \cap S_{lat}}^2 \lesssim 1 + \eps C.
 \end{aligned}
\end{equation}
Then the same estimate holds for $v_k^+$. 

\

b) Assume the bilinear bounds
\begin{equation}\label{w-bootstrap}
\begin{aligned}
d_{j}^{-1} d_{k}^{-1} 2^{-\frac{j}{2}} 2^{\frac{\alpha}{2}\max(j, k)} \|\tilde w_{j}^+ \tilde \psi_{k}^+\|_{L^2} \lesssim 1 + \eps C.
 \end{aligned}
\end{equation}
Then the same estimate holds for $v_{j}^+ \phi_{k}^+$.

\

c) Given the bootstrap assumption \eqref{linBootstrapAssump}, we have
\begin{equation}\label{w-data}
\|\tilde w_k^+(0)\|_{H^{s - \half}} \lesssim c_k.
\end{equation}
\end{lemma}

\begin{proof}
$a)$ From the definition above \eqref{finalNFlin}, we have
\[
\tilde w_k^+ = w_k^+ - B_k^{lin}(\phi, v).
\]
Lemma~\ref{l:Bk-bootstraplin} provides the sufficient bound for $B_k^{lin}$, using that 
\[
d_k^{-1} \|B_k^{lin}(\phi, v)\|_{S \cap S_{lat}} \lesssim 2^{\frac14(1 - \alpha)k} C \eps \ll C \eps
\]
for large $2^k$.

\

$b)$ We have
\[
\tilde w_{j}^+ \tilde \psi_{k}^+ = (w_j^+ - B_j^{lin}(\phi, v)) (\psi_k^+ - B_k(\phi, \phi)).
\]
We estimate the cubic and higher terms to first reduce to $w_j^+ \psi_k^+$. Consider in particular 
\[
w_j^+B_k(\phi, \phi)
\]
with the other cases being similar. 

We first consider the case $j > k$, which we in turn reduce to two cases. When the inputs are also at frequency $\geq j$, we use the bootstrap assumption \eqref{linBootstrapAssump} with Proposition~\ref{p:expBd} on $w_j^+$, and $b)$ of Lemma~\ref{l:Bk-bootstrap} on $B_k$ (adapted to Theorem~\ref{t:nonlinear} instead of the bootstrap estimate, so we may drop the dependence on $C$):
\begin{equation*}
\begin{aligned}
\|w_j^+B_k(\phi_{\geq j}, \phi_{\geq j})\|_{L^2} \lesssim \|w_j^+\|_{L^4L^\infty} \|B_k(\phi_{\geq j}, \phi_{\geq j})\|_{S} &\lesssim 2^{\frac{j}{2}} 2^{\frac14(1 - \alpha)j}C d_j \cdot 2^{-\frac14(1 + \alpha)j}\eps^2 d_k\\
&\lesssim 2^{\frac{j}{2}}  2^{-\frac{\alpha}{2}j} C\eps^2 d_jd_k
\end{aligned}
\end{equation*}
as desired.

It remains to estimate 
\begin{equation}\label{wbjest}
w_j^+B_k(\phi_{< j}, \phi_{< j}),
\end{equation}
for which we use the bilinear bootstrap estimate. Similar to the fully nonlinear setting, we use Proposition~\ref{p:expBd} to apply bilinear estimates though the $e^{iA}$ conjugations, along with the gain from $b)$ of Lemma~\ref{l:Bk-bootstraplin} to compensate for the loss from the pointwise estimate on the remaining third variable:
\begin{equation*}
\begin{aligned}
\|w_j^+B_k(\phi_{\ell}, \phi_{\ell})\|_{L^2} &= 2^{-\alpha \ell}\|w_j^+e^{iA} L_k(\phi_{\ell}, \phi_{\ell})\|_{L^2} \\
&\lesssim 2^{-\alpha \ell} \|L(w_j, \phi_{\ell}, \phi_{\ell})\|_{L^2} \\
&\lesssim 2^{-\alpha \ell}\cdot 2^{\frac{j}{2}} 2^{-\frac{\alpha}{2}j} C\eps d_jd_{\ell} \cdot 2^{\frac14(3\alpha - 1)\ell}\eps c_{\ell} \\
&\lesssim 2^{\frac{j}{2}} 2^{-\frac{\alpha}{2}j}C\eps^2 d_jd_{\ell} \cdot 2^{-\frac14(1 + \alpha)\ell}.
\end{aligned}
\end{equation*}
Then summing, we obtain
\[
\|\psi_j^+B_k(\phi_{k < \cdot < j}, \phi_{k < \cdot < j})\|_{L^2} \lesssim 2^{\frac{j}{2}} 2^{-\frac{\alpha}{2}j}C\eps^2 d_jd_{k}
\]
as desired. A similar bilinear estimate for 
\[
w_j^+B_k(\phi_{k}, \phi_{< k}), \qquad w_j^+B_k(\phi_{< k}, \phi_{k})
\]
addresses the remaining components of \eqref{bjest}. A similar case analysis using the bilinear bootstrap estimate can be applied to address case when $j \leq k$. 

To conclude, we use Proposition~\ref{p:expBd} to reduce from $w_j^+w_k^+$ to $v_j^+v_k^+$.


\

$c)$ We use part $c)$ of Lemma~\ref{l:Bk-bootstraplin} to estimate $B_k$ and Proposition~\ref{p:expBd} to absorb $e^{iA}$.

\end{proof}

\subsection{Bounds on the source terms}

In this section, we bound the cubic and higher source terms $Q_k^{lin}$ of the equation \eqref{wfinalError} for $w_k^+$.

\begin{lemma}\label{l:source-bdlin}
The source terms $Q_k^{lin}$ of \eqref{wfinalError} satisfy the estimate
\[
\|Q_k^{lin}\|_{L^1_tL^2_x} \lesssim 2^{\frac{k}{2}} C\eps d_k.
\]
\end{lemma}

\begin{proof}

We consider the terms in $Q_k^{lin}$. We first discuss the cubic terms, starting with $Q_k^{exp, 3+}(v_k^+)$, which is provided in Lemma~\ref{l:expCom}. Here we demonstrate an estimate for the third term of $Q_k^{exp, 3+}(v_k^+)$, which is typical:
\[
|D|^{-\alpha}  e^{iA} \D_x P_{(k', k)} (\phi^2) v_k^+ = |D|^{-\alpha}  e^{iA} \D_x L_k(\phi, \phi, v_k^+).
\]
Using Proposition~\ref{p:expBd}, it suffices to show
\begin{equation}\label{source-est-lin}
\|L_k(\phi, \phi, v_k^+) \|_{L^1_tL^2_x} \lesssim 2^{\frac{k}{2}} 2^{k(\alpha - 1)} C\eps d_k.
\end{equation}

\

For the low frequency component on the first two inputs, we use Strichartz twice:
\begin{equation*}
\begin{aligned}
\|L_k(\phi_{\leq k}, \phi_{\leq k}, v_k^+)\|_{L^1_tL^2_x} &\lesssim \|\phi_{\leq k}\|_{L^4_t L^\infty_x}^2 \|v_k^+\|_{L^\infty_t L^2_x} \lesssim \eps^2 2^{-(1 - \alpha)k} \cdot 2^{\frac{k}{2}}C d_k
\end{aligned}
\end{equation*}
which suffices.

For the remaining high frequency components, let $j > k$ and consider
\[L_k(\phi_{j}, \phi_{j}, v_k^+).\]
When $\alpha \leq \frac53$, we use bilinear estimates on the latter two variables and Strichartz for the first:
\begin{equation*}
\begin{aligned}
\|L_k(\phi_{j}, \phi_{j}, v_k^+)\|_{L^1_tL^2_x} &\lesssim  2^{-\half(1 - \alpha)j} \eps c_{j} \cdot 2^{\frac{k}{2}} 2^{-\frac{\alpha}{2} j}C\eps d_{j} d_k \\
&\lesssim 2^{\frac{k}{2}} 2^{-\half j} 2^{-\frac34(1 - \alpha)j} C\eps^2 d_k \\
&= 2^{\frac{k}{2}} 2^{(-\frac54 + \frac34\alpha) j} C\eps^2 d_k \\
&\lesssim 2^{\frac{k}{2}} C\eps^2 d_k.
\end{aligned}
\end{equation*}
When $\alpha \geq \frac54$, and in particular when $\alpha > \frac53$, we use lateral Strichartz estimates, estimating $v_k^+$ in $L^4_x L^\infty_t$, and both instances of $\phi_j$ in $L^\infty_x L^2_t$. Precisely,
\begin{equation*}
\begin{aligned}
\|L_k(\phi_{j}, \phi_{j}, v_k^+)\|_{L^1_tL^2_x} &\lesssim \|L_k(\phi_{j}, \phi_{j}, v_k^+)\|_{L^2_{t, x}} \\
&\lesssim  2^{\frac12(\alpha - 3)j} (\eps c_j)^2 \cdot 2^{\frac{k}{2}} 2^{\frac{k}{4}} C d_k \\
&\lesssim 2^{\frac{k}{2}} 2^{\frac12(\alpha - 3)j} 2^{\frac{k}{4}} C\eps^2 d_k \\
&\lesssim 2^{\frac{k}{2}} 2^{k(\alpha - 1)}C\eps^2 d_k.
\end{aligned}
\end{equation*}

The other cubic terms of $Q_k^{lin}$ are treated similarly using Proposition~\ref{p:expBd} to treat the exponential conjugations, and Proposition~\ref{p:Bk-est} to estimate the bilinear forms. We also use Lemma~\ref{l:expCom} to see that $[(\D_t - |D|^\alpha \D_x), e^{iA}]$ is an operator of order 1.

\

Lastly, the quadratic term of $Q_k^{lin}$ is estimated using Proposition~\ref{p:expBd} followed by a bilinear estimate,
\[
\| e^{iA}( v_{< 0} \D_x \phi_k^+)\|_{L^1_tL^2_x} \lesssim 2^{-\frac{\alpha}{2} k}C\eps d_k
\]
which more than suffices.

\end{proof}

\subsection{Closing the bootstrap}

To complete the proof of the bootstrap estimate \eqref{linBootstrap}, it remains to prove the estimates \eqref{w-bootstrap-strich} and \eqref{w-bootstrap} on $\tilde w_k^+$. For the former, we apply Theorem~\ref{t:strichartz-transport} and a straightforward energy estimate to obtain
\[
\|\tilde w_k^+\|_{S \cap S_{lat}} \lesssim \|\tilde w_k^+(0)\|_{L^2} + \|Q_k^{lin}\|_{L^1L^2}.
\]
Using \eqref{w-data} for $\tilde w_k^+(0)$ and Lemma~\ref{l:source-bdlin} for $Q_k^{lin}$, we obtain
\begin{equation}\label{psi-bootstrap-strich-pf-lin}
\|\tilde w_k^+\|_{S \cap S_{lat}} \lesssim 2^{\frac{k}{2}} d_k + 2^{\frac{k}{2}} C\eps d_k
\end{equation}
as desired.

For the bilinear estimate \eqref{w-bootstrap}, we likewise apply Proposition~\ref{p:bilinear} to obtain
\begin{equation*}
\begin{aligned}
 \|\tilde w_j^+\tilde \psi_k^+\|_{L^2} &\lesssim 2^{-\frac{\alpha}{2}\max(j, k)}(\|\tilde w_j^+\|_{L^\infty L^2}\|\tilde \psi_k^+\|_{L^\infty L^2)} \\
&\quad + \|Q_j^{lin}\|_{L^1L^2}\|\tilde \psi_k^+\|_{L^\infty L^2}  + \|Q_k\|_{L^1L^2}\|\tilde w_j^+\|_{L^\infty L^2}). 
\end{aligned}
\end{equation*}
Then using \eqref{psi-bootstrap-strich-pf-lin} and Lemma~\ref{l:source-bdlin}, we obtain
\[
 \|\tilde w_j^+\tilde \psi_k^+\|_{L^2} \lesssim 2^{\frac{j}{2}} 2^{-\frac{\alpha}{2}\max(j, k)}(\eps + C \eps^2  + \eps^3)d_j d_k
\]
as desired.

\section{Local well-posedness}

The local well-posedness for \eqref{BO}, stated in Theorem~\ref{t:lwp}, follows from the a priori estimates of Theorems~\ref{t:nonlinear} and \ref{t:linear} for the nonlinear and linearized equations, respectively. The argument does not differ substantially from the case of the classical Benjamin-Ono equation with $\alpha = 1$ (see \cite{ifrim2017well}). For completeness, we record the main steps here.

By scaling, we assume small initial data $\phi_0 \in H^s$, $s > \frac34(1 - \alpha)$, satisfying
\begin{equation}\label{lwp-eps}
\|\phi_0\|_{H^s} \leq \eps \ll 1.
\end{equation}
We consider the sequence of regularized data $\phi^{(n)}(0) = P_{< n} \phi_0$, which uniformly admits a frequency envelope
\[
\|P_k \phi^{(n)}(0)\|_{H^s} \leq \eps c_k.
\]
Then by Theorem~\ref{t:nonlinear}, the corresponding solutions $\phi^{(n)}$ exist and uniformly satisfy the bound
\[
\|P_k \phi^{(n)}\|_{S^{s}} \leq \eps c_k,
\]
where we use the notation
\[
S^s = \langle D \rangle^{-s} S.
\]

On the other hand, using the bounds for the linearized equation from Theorem~\ref{t:linear}, the differences satisfy
\[
\|\phi^{(n)} - \phi^{(m)}\|_{S^{s - \half}} \lesssim (2^{-\frac{n}{2}} + 2^{-\frac{m}{2}}) \eps.
\]
As a result, the sequence $\phi^{(n)}$ converges to some function $\phi \in S^{s - \half}$, which satisfies
\begin{equation}\label{highfreq}
\|P_k \phi\|_{S^s} \leq \eps c_k.
\end{equation}
Further, we have convergence in $\ell^2_k S^s$, as follows. For fixed $k$,
\[
\limsup_{n \rightarrow \infty} \|\phi^{(n)} - \phi\|_{\ell^2 S^s} \leq \|P_{\geq k} \phi\|_{\ell^2 S^s} + \limsup_{n \rightarrow \infty} \|P_{\leq k} (\phi^{(n)} - \phi)\|_{\ell^2 S^s}  + \|P_{\geq k} \phi^{(n)}\|_{\ell^2 S^s} \lesssim c_{\geq k}.
\]
Letting $k \rightarrow \infty$, we obtain 
\[
\lim_{n \rightarrow \infty} \|\phi^{(n)} - \phi\|_{\ell^2 S^s} = 0,
\]
which implies that $\phi$ satisfies \eqref{BO} in the sense of distributions.

We next show continuous dependence on data in $H^s$. Consider a sequence of data $\phi^{(n)}(0)$ satisfying \eqref{lwp-eps} uniformly, and such that
\[
\lim_{n \rightarrow \infty} \|\phi^{(n)}(0) - \phi_0\|_{H^s} = 0.
\]
We again use the decomposition 
\[
\phi^{(n)} - \phi = - P_{\geq k} \phi + P_{\leq k} (\phi^{(n)} - \phi) - P_{\geq k} \phi^{(n)},
\]
where the contributions from the first and second terms vanish after taking $k \rightarrow \infty$ by the frequency envelope bound \eqref{highfreq} and the weak Lipschitz dependence, respectively. It remains to show
\[
\lim_{k \rightarrow \infty} \limsup_{n \rightarrow \infty} \|P_{\geq k} \phi^{(n)}\|_{\ell^2 S^s} = 0.
\]

Given $\delta > 0$, we have
\[
\|\phi^{(n)}(0) - \phi_0\|_{H^s} \leq \delta, \qquad n \geq N(\delta).
\]
Let $\eps c_k$ be an $H^s$ frequency envelope for $\phi_0$, and $\delta c_k(n)$ be an $H^s$ frequency envelope for $\phi^{(n)}(0) - \phi_0$. Then $\eps c_k + \delta c_k(n)$ is an $H^s$ frequency envelope for $\phi^{(n)}(0)$, and by \eqref{highfreq} we obtain for $n \geq N(\delta)$,
\[
\|P_{\geq k} \phi^{(n)}\|_{\ell^2 S^s} \lesssim \eps c_{\geq k} + \delta c_{\geq k}(n) \lesssim \eps c_{\geq k} + \delta.
\]
Thus 
\[
\lim_{k \rightarrow \infty} \limsup_{n \rightarrow \infty} \|P_{\geq k} \phi^{(n)}\|_{\ell^2 S^s} \lesssim \delta
\]
with arbitrary $\delta > 0$, concluding the proof.

\bibliography{allbib}

\begin{thebibliography}{10}

\bibitem{capillary}
Albert Ai.
\newblock Improved low regularity theory for gravity-capillary waves, 2023.

\bibitem{alazard2011strichartz}
Thomas Alazard, Nicolas Burq, and Claude Zuily.
\newblock Strichartz estimates for water waves.
\newblock {\em Ann. Sci. \'{E}c. Norm. Sup\'{e}r. (4)}, 44(5):855--903, 2011.

\bibitem{FBId}
Jean-Marc Delort.
\newblock {\em F.{B}.{I}. transformation}, volume 1522 of {\em Lecture Notes in
  Mathematics}.
\newblock Springer-Verlag, Berlin, 1992.
\newblock Second microlocalization and semilinear caustics.

\bibitem{herr}
Sebastian Herr.
\newblock Well-posedness for equations of {B}enjamin-{O}no type.
\newblock {\em Illinois J. Math.}, 51(3):951--976, 2007.

\bibitem{herr2010differential}
Sebastian Herr, Alexandru~D Ionescu, Carlos~E Kenig, and Herbert Koch.
\newblock A para-differential renormalization technique for nonlinear
  dispersive equations.
\newblock {\em Communications in Partial Differential Equations},
  35(10):1827--1875, 2010.

\bibitem{ifrim2022dispersive}
Mihaela Ifrim, Herbert Koch, and Daniel Tataru.
\newblock Dispersive decay of small data solutions for the kdv equation, 2022.

\bibitem{ifrim2017well}
Mihaela Ifrim and Daniel Tataru.
\newblock Well-posedness and dispersive decay of small data solutions for the
  benjamin-ono equation.
\newblock {\em arXiv preprint arXiv:1701.08476}, 2017.

\bibitem{ITprimer}
Mihaela {Ifrim} and Daniel {Tataru}.
\newblock {Local well-posedness for quasilinear problems: a primer}.
\newblock {\em arXiv e-prints}, page arXiv:2008.05684, August 2020.

\bibitem{IKbo}
Alexandru~D. Ionescu and Carlos~E. Kenig.
\newblock Global well-posedness of the {B}enjamin-{O}no equation in
  low-regularity spaces.
\newblock {\em J. Amer. Math. Soc.}, 20(3):753--798, 2007.

\bibitem{KLVsharp}
Rowan {Killip}, Thierry {Laurens}, and Monica {Visan}.
\newblock {Sharp well-posedness for the Benjamin--Ono equation}.
\newblock {\em arXiv e-prints}, page arXiv:2304.00124, March 2023.

\bibitem{marzuola2008wave}
Jeremy Marzuola, Jason Metcalfe, and Daniel Tataru.
\newblock Wave packet parametrices for evolutions governed by {PDO}'s with
  rough symbols.
\newblock {\em Proceedings of the American Mathematical Society},
  136(2):597--604, 2008.

\bibitem{BOill}
L.~Molinet, J.~C. Saut, and N.~Tzvetkov.
\newblock Ill-posedness issues for the {B}enjamin-{O}no and related equations.
\newblock {\em SIAM J. Math. Anal.}, 33(4):982--988, 2001.

\bibitem{MPcauchy}
Luc Molinet and Didier Pilod.
\newblock The {C}auchy problem for the {B}enjamin-{O}no equation in {$L^2$}
  revisited.
\newblock {\em Anal. PDE}, 5(2):365--395, 2012.

\bibitem{molinet2018well}
Luc Molinet, Didier Pilod, and St{\'e}phane Vento.
\newblock On well-posedness for some dispersive perturbations of {B}urgers'
  equation.
\newblock In {\em Annales de l'Institut Henri Poincar{\'e} C, Analyse non
  lin{\'e}aire}. Elsevier, 2018.

\bibitem{Tlowreg}
Blaine Talbut.
\newblock Low regularity conservation laws for the {B}enjamin-{O}no equation.
\newblock {\em Math. Res. Lett.}, 28(3):889--905, 2021.

\bibitem{Ltao}
Terence Tao.
\newblock Global regularity of wave maps. {II}. {S}mall energy in two
  dimensions.
\newblock {\em Comm. Math. Phys.}, 224(2):443--544, 2001.

\bibitem{taoBO}
Terence Tao.
\newblock Global well-posedness of the {B}enjamin-{O}no equation in {$H^1({\bf
  R})$}.
\newblock {\em J. Hyperbolic Differ. Equ.}, 1(1):27--49, 2004.

\bibitem{tao}
Terence Tao.
\newblock Global well-posedness of the {B}enjamin-{O}no equation in {$H^1({\bf
  R})$}.
\newblock {\em J. Hyperbolic Differ. Equ.}, 1(1):27--49, 2004.

\bibitem{taobook}
Terence Tao.
\newblock {\em Nonlinear dispersive equations}, volume 106 of {\em CBMS
  Regional Conference Series in Mathematics}.
\newblock Conference Board of the Mathematical Sciences, Washington, DC; by the
  American Mathematical Society, Providence, RI, 2006.
\newblock Local and global analysis.

\bibitem{tataru2004phase}
Daniel Tataru.
\newblock Phase space transforms and microlocal analysis.
\newblock {\em Phase space analysis of partial differential equations},
  2:505--524, 2004.

\end{thebibliography}
\bibliographystyle{plain}

\end{document}